\documentclass[a4paper,11pt]{amsart}
\usepackage[left=2.7cm,right=2.7cm,top=3.5cm,bottom=3cm]{geometry}

\usepackage{graphicx}
\newcommand{\Bmu}{\mbox{$\raisebox{-0.59ex}
  {$l$}\hspace{-0.18em}\mu\hspace{-0.88em}\raisebox{-0.98ex}{\scalebox{2}
  {$\color{white}.$}}\hspace{-0.416em}\raisebox{+0.88ex}
  {$\color{white}.$}\hspace{0.46em}$}{}}

\usepackage{amsthm,amssymb,amsmath,amsfonts,mathrsfs,amscd,dsfont}
\usepackage[latin1]{inputenc}
\usepackage[all,cmtip]{xy}
\usepackage{latexsym}
\usepackage{longtable}
\usepackage{mathtools}

\numberwithin{equation}{section}

\usepackage[pagebackref]{hyperref}

\mathtoolsset{showonlyrefs}

\newfont{\cyr}{wncyr10 scaled 1100}
\newfont{\cyrr}{wncyr9 scaled 1000}

\theoremstyle{plain}
\newtheorem{theorem}{Theorem}[section]
\newtheorem*{theoremA}{Theorem A}
\newtheorem*{theoremB}{Theorem B}
\newtheorem{proposition}[theorem]{Proposition}
\newtheorem{lemma}[theorem]{Lemma}
\newtheorem{corollary}[theorem]{Corollary}
\newtheorem{conjecture}[theorem]{Conjecture}

\newcommand{\defeq}{\vcentcolon=}

\theoremstyle{definition}
\newtheorem{definition}[theorem]{Definition}
\newtheorem{assumption}[theorem]{Assumption}

\theoremstyle{remark}
\newtheorem{remark}[theorem]{Remark}

\newtheorem{notationt}[theorem]{Notation and terminology}

\newcommand{\Q}{\mathds Q}
\newcommand{\N}{\mathds N}
\newcommand{\Z}{\mathds Z}

\newcommand{\C}{\mathds C}

\newcommand{\F}{\mathds F}
\newcommand{\T}{\mathds T}
\newcommand{\V}{\mathds V}
\newcommand{\GG}{\mathds G}

\DeclareMathOperator{\Spec}{Spec}
\DeclareMathOperator{\Pic}{Pic}
\DeclareMathOperator{\End}{End}

\DeclareMathOperator{\Frob}{Frob}

\DeclareMathOperator{\Hom}{Hom}

\DeclareMathOperator{\Gal}{Gal}
\DeclareMathOperator{\GL}{GL}

\DeclareMathOperator{\SL}{SL}
\DeclareMathOperator{\Sel}{Sel}

\DeclareMathOperator{\CH}{CH}

\DeclareMathOperator{\im}{im}

\newcommand{\res}{\mathrm{res}}
\newcommand{\cores}{\mathrm{cores}}

\newcommand{\tr}{\mathrm{tr}}
\newcommand{\ord}{\mathrm{ord}}
\newcommand{\new}{\mathrm{new}}

\newcommand{\cont}{\mathrm{cont}}

\newcommand{\an}{\mathrm{an}}
\newcommand{\alg}{\mathrm{alg}}
\newcommand{\cla}{\mathrm{cl}}

\newcommand{\Sha}{\mbox{\cyr{X}}}

\usepackage[usenames]{color}
\definecolor{Indigo}{rgb}{0.2,0.1,0.7}
\definecolor{Violet}{rgb}{0.5,0.1,0.7}
\definecolor{White}{rgb}{1,1,1}
\definecolor{Green}{rgb}{0.1,0.9,0.2}

\newcommand{\longmono}{\mbox{\;$\lhook\joinrel\longrightarrow$\;}}

\newcommand{\longepi}{\mbox{\;$\relbar\joinrel\twoheadrightarrow$\;}}

\newcommand{\E}{\mathcal E}

\newfont{\gotip}{eufb10 at 12pt}

\newcommand{\cO}{{\mathcal O}}

\newcommand{\spe}{\mathrm{sp}}

\newcommand{\m}{\mathfrak{m}}
\newcommand{\p}{\mathfrak{p}}
\newcommand{\fP}{\mathfrak{P}}

\newcommand{\hf}{\boldsymbol f^{(p)}}

\DeclareMathOperator{\GS}{GS}
\DeclareMathOperator{\Ta}{Ta}

\DeclareMathOperator{\rank}{rank}

\DeclareMathOperator{\f}{\boldsymbol f}

\newcommand{\KK}{\mathscr{K}}

\include{thebibliography}

\begin{document}

\title[On Shafarevich--Tate groups and analytic ranks in Coleman families]{On Shafarevich--Tate groups and analytic ranks\\in families of modular forms, II. Coleman families}
%\today
%\date{}
\author{Maria Rosaria Pati, Gautier Ponsinet and Stefano Vigni}

\thanks{}

\begin{abstract}
This is the second article in a two-part project whose aim is to study algebraic and analytic ranks in $p$-adic families of modular forms. Let $f$ be a newform of weight $2$, square-free level $N$ and trivial character, let $A_f$ be the abelian variety attached to $f$, whose dimension will be denoted by $d_f$, and for every prime number $p\nmid N$ let $\hf$ be a $p$-adic Coleman family through $f$ over a suitable open disc in the $p$-adic weight space. We prove that, for all but finitely many primes $p$ as above, if $A_f(\Q)$ has rank $r\in\{0,d_f\}$ and the $p$-primary part of the Shafarevich--Tate group of $A_f$ over $\Q$ is finite, then all classical specializations of $\hf$ of weight congruent to $2$ modulo $2(p-1)$ and trivial character have finite $p$-primary Shafarevich--Tate group and $r/d_f$-dimensional image of the relevant $p$-adic \'etale Abel--Jacobi map. As a second contribution, assuming the non-degeneracy of certain height pairings \emph{\`a la} Gillet--Soul\'e between Heegner cycles, we show that, for all but finitely many $p$, if $f$ has analytic rank $r\in\{0,1\}$, then all classical specializations of $\hf$ of weight congruent to $2$ modulo $2(p-1)$ and trivial character have analytic rank $r$. This result provides some evidence for a conjecture of Greenberg on analytic ranks in families of modular forms.
\end{abstract}

\address{Laboratoire de Math\'ematiques Nicolas Oresme (LMNO), UMR 6139, Normandie Universit\'e, Universit\'e de Caen Normandie -- CNRS, 14000 Caen, France}
\email{maria-rosaria.pati@unicaen.fr}

\address{Mathematisches Institut, Universit\"at Heidelberg, Im Neuenheimer Feld 205, 69120 heidelberg, Germany}
\email{gponsinet@mathi.uni-heidelberg.de}

\address{Dipartimento di Matematica, Universit\`a di Genova, Via Dodecaneso 35, 16146 Genova, Italy}
\email{stefano.vigni@unige.it}



\maketitle

%\tableofcontents

\section{Introduction}

The present article completes a two-part project whose goal is to study certain arithmetic invariants of modular forms (algebraic ranks, analytic ranks, Shafarevich--Tate groups) when the modular forms they are attached to vary in a $p$-adic family. More precisely, we extend the results of \cite{Vigni}, which deals with Hida (\emph{i.e.}, slope $0$) families, to Coleman (\emph{i.e.}, finite slope) families. A general introduction, in the context of Hida families, to the questions that are studied in these two papers is given in \cite{Vigni}, so here we content ourselves with describing the main results of this article.

Let $f$ be a newform of weight $2$, square-free level $N$ and trivial character, and for every prime number $p\nmid N$ let $\hf_U$ be a $p$-adic Coleman family passing through $f$ (or, rather, through a suitable $p$-stabilization of $f$) over an open disc $U$ in the $p$-adic weight space $\mathscr W_N$ of tame level $N$. In this paper we are exclusively interested in the specializations $f_k$ of $\hf_U$ at even integers $k\in U\cap\mathscr W_{N,\cla}$, where $\mathscr W_{N,\cla}$ denotes the subset of $\mathscr W_N$ consisting of classical weights. We warn the reader that in this introduction we ignore the phenomenon of ``$p$-stabilization'' that we alluded to before; thus, whenever we speak of an invariant associated with $f_k$ we tacitly understand that this notion refers to the newform of level $N$ whose $p$-stabilization coincides with $f_k$.

Let $\bar\Z$ denote the ring of integers in a fixed algebraic closure $\bar\Q$ of $\Q$ and choose a prime ideal $\fP$ of $\bar\Z$ above $p$. Furthermore, write $\cO_{\Q_{f_k}}$ for the ring of integers of the Hecke field of $f_k$ and $\cO_{\Q_{f_k},\fP}$ for its completion at $\fP$. To every specialization $f_k$ of $\hf_U$ as above one can attach, among others, three basic arithmetic invariants:
\begin{itemize}
\item its \emph{algebraic $\fP$-rank}, \emph{i.e.}, the rank $r_{\alg,\fP}(f_k)$ over $\cO_{\Q_{f_k},\fP}$ of the image of the $\fP$-adic \'etale Abel--Jacobi map associated with $f_k$;
\item its \emph{$\fP$-primary Shafarevich--Tate group} $\Sha_\fP(f_k/\Q)$; 
\item its \emph{analytic rank}, \emph{i.e.}, the order of vanishing $r_\an(f_k)$ at $s=k/2$ of the complex $L$-function of $f_k$.
\end{itemize}
Clearly, the first two invariants are algebraic, whereas the third is of an analytic nature. These invariants are expected to encode a huge amount of information on the arithmetic of the corresponding modular form, as epitomized by the conjectures of Birch--Swinnerton-Dyer (in weight $2$) and of Beilinson--Bloch--Kato (in higher weight). Our goal in this paper is to study the behaviour of these invariants as $k$ varies over (a suitable subset of) the classical even weights in the open disc $U$ over which our Coleman family is defined.

Let $A_f$ be the abelian variety over $\Q$ of $\GL_2$-type attached to $f$ by the Eichler--Shimura construction, write $d_f$ for the dimension of $A_f$ and let $\Sha_\fP(A_f/\Q)$ be the $\fP$-primary part of the Shafarevich--Tate group of $A_f$ over $\Q$. Our first main result, which concerns algebraic ranks and Shafarevich--Tate groups, is the following.

\begin{theoremA} 
Suppose that the rank of $A_f(\Q)$ is $r\in\{0,d_f\}$. For all but finitely many primes $p$, if $\Sha_\fP(A_f/\Q)$ is finite, then all specializations $f_k$ of $\hf_U$ with $k\in U\cap\mathscr W_{N,\cla}$ such that $k\geq4$ and $k\equiv2\pmod{2(p-1)}$ satisfy $r_{\alg,p}(f_k)=r/d_f$ and $\#\Sha_\fP(f_k/\Q)<\infty$. 
\end{theoremA}

This result corresponds in the text to Theorem \ref{sha-thm} if $\mathrm{rank}_\Z\, A_f(\Q)=d_f$ and to Theorem \ref{sha-thm2} if $\mathrm{rank}_\Z\, A_f(\Q)=0$.

As far as our second main contribution is concerned, one of the motivations behind it is provided by a conjecture of Greenberg (\cite{Greenberg-CRM}), a special case of which predicts that the analytic ranks of the even weight specializations $f_k$ of $\hf_U$ should be as small as allowed by the functional equation, with at most finitely many exceptions. We note that, as for the analogous result in \cite{Vigni} for Hida families, here we need to assume the non-degeneracy of certain height pairings \emph{\`a la} Gillet--Soul\'e that have been introduced by S.-W. Zhang in \cite{Zhang-heights} to prove a counterpart for higher (even) weight modular forms of the Gross--Zagier formula. Our result can be stated as follows.

\begin{theoremB} 
Suppose that the analytic rank of $f$ is $r\in\{0,1\}$ and that the height pairing in Zhang's formula is non-degenerate. For all but finitely many primes $p$, all specializations $f_k$ of $\hf_U$ with $k\in U\cap\mathscr W_{N,\cla}$ such that $k\geq4$ and $k\equiv2\pmod{2(p-1)}$ have analytic rank $r$.
\end{theoremB}

It is worth remarking that this theorem, which we prove in \S \ref{B-1-subsubsec} if $r_\an(f)=1$ and in \S \ref{B-0-subsubsec} if $r_\an(f)=0$, offers some evidence for the aforementioned conjecture by Greenberg.

Our strategy for proving Theorems A and B, which is inspired by \cite{Vigni}, goes as follows. Under the assumption that $r_\an(f)\in\{0,1\}$, we introduce two sets $\Xi_f$, $\Omega_f$ of prime numbers (\S \ref{choice-subsubsec}, \S \ref{newform-0-subsubsec}). Both sets, which are defined in terms of, among other conditions, the non-triviality modulo $p$ of the imaginary quadratic Heegner point on $A_f$ appearing in the Gross--Zagier formula, consist of all but finitely many primes. Using recent results of B\"uy\"ukboduk--Lei on the interpolation of (generalized) Heegner cycles in Coleman families (\cite{BL}), we obtain our key technical result: for any $k\in U\cap\mathscr W_{N,\cla}$ such that $k\geq4$ and $k\equiv2\pmod{2(p-1)}$, the imaginary quadratic Heegner cycle $y_k$ (to be denoted by $y_{k,K}$ in the main body of the paper) that was originally defined by Nekov\'a\v{r} is non-torsion over $\cO_{\Q_{f_k},\fP}$ in the relevant \'etale Abel--Jacobi image (\S \ref{nontrivial-subsubsec}). Here we would like to bring to the reader's attention a serious extra difficulty arising in the finite slope setting in comparison with similar arguments that are proposed in \cite{Vigni} for ordinary families. Namely, unlike what happens in the work of Howard (\cite{Howard-Inv}), Castella (\cite{CasHeeg}) and Ota (\cite{ota-JNT}) on the interpolation of Heegner points and Heegner cycles in Hida families, B\"uy\"ukboduk and Lei have to face delicate (non-)integrality issues of their big Heegner-type classes in Coleman families; since the integrality of these classes is a key ingredient for our goals, we carefully explain why it holds unconditionally in our setting (\S \ref{integrality-subsubsec}). Once the non-degeneracy of Zhang's heights is assumed, Theorem B for all $p\in\Xi_f$ if $r_\an(f)=1$ or for all $p\in\Omega_f$ if $r_\an(f)=0$ is then a consequence of Zhang's formula of Gross--Zagier type for modular forms (Section \ref{analytic-sec}).

Finally, set $r\defeq\rank_\Z A_f(\Q)$. In order to prove Theorem A, we note that the assumption that $r\in\{0,d_f\}$ and $\#\Sha_\fP(A_f/\Q)<\infty$ amounts, thanks to converses to the Kolyvagin--Gross--Zagier theorem that are due to Castella--\c{C}iperiani--Skinner--Sprung (if $r=0$, \cite{CCSS}) and to Sweeting (if $r=d_f$, \cite{sweeting}), to the condition $r_\an(f)=r/d_f$. Since $y_k$ is not torsion, Theorem A for all $p\in\Xi_f$ if $r=d_f$ or for all $p\in\Omega_f$ if $r=0$ follows (\S \ref{sha-subsec} and \S \ref{sha-subsec2}) from Nekov\'a\v{r}'s results on the arithmetic of Chow groups of Kuga--Sato varieties (\cite{Nek}) combined with a comparison, which can be found in \cite{Vigni}, of \'etale Abel--Jacobi images over $\Q$ and over certain imaginary quadratic fields.

We close this introduction by pointing out that in Appendix \ref{appendix} we explain how the two-variable $p$-adic $L$-functions that have been attached to $p$-adic Coleman families by Bella\"iche--Pollack--Stevens (\cite{bellaiche}) and Panchishkin (\cite{Panchishkin-Inventiones}), which play no role in our strategy towards Theorems A and B, can be used to obtain a refinement of the rank $0$ part of Theorem B for which the congruence condition on the weights is not necessary and no assumption on the height pairings needs to be imposed. While a result of this kind is well known to hold for Hida families (see, \emph{e.g.}, \cite[Theorem 7]{Howard-derivatives}), no analogue for Coleman families seems to be available, so Theorem \ref{main-appendix-thm} fills a gap in the literature and may be of independent interest.

\subsection{Notation and conventions} \label{notation-subsec}

We denote by $\bar\Q$ an algebraic closure of $\Q$ and write $\bar\Z$ for the ring of integers in $\bar\Q$ (\emph{i.e.}, the integral closure of $\Z$ in $\bar\Q$). For every prime number $\ell$ we fix an algebraic closure $\bar\Q_\ell$ of $\Q_\ell$ and denote by $\C_\ell$ the completion of $\bar\Q_\ell$. Moreover, for every prime $\ell$ and every number field $F$ we also fix field embeddings
\[ \iota_\ell:\bar\Q\longmono\bar\Q_\ell,\quad\iota_F:F\longmono\bar\Q. \]
%The map $\iota_\ell$ determines a prime ideal $\mathfrak L$ of $\bar\Z$ above $\ell$ that, in turn, induces a prime $\mathfrak L_F\defeq\iota_F^{-1}\bigl(\mathfrak L\cap\iota_F(F)\bigr)$ of $F$ above $\ell$. In order to simplify our notation, when there is no risk of confusion we will often use alternative symbols to denote the ideal $\mathfrak L_F$ and related objects. For example, we write $F_\mathfrak L$ in place of $F_{\mathfrak L_F}$ for the completion of $F$ at the prime $\mathfrak L_F$.
For any number field $K$, we denote by $G_K\defeq\Gal(\bar K/K)$ the absolute Galois group of $K$, where $\bar K$ is a fixed algebraic closure of $K$. For any continuous $G_K$-module $M$ we write $H^i(K,M)$ for the $i$-th continuous cohomology group of $G_K$ with coefficients in $M$ in the sense of Tate (\cite[\S 2]{Tate}). Finally, if $K/F$ is an extension of number fields, then  
\[ \res_{K/F}:H^i(F,M)\longrightarrow H^i(K,M),\quad\cores_{K/F}:H^i(K,M)\longrightarrow H^i(F,M) \] 
denote the restriction and corestriction maps in cohomology, respectively. 

\subsection*{Acknowledgements} 

It is a pleasure to thank Kazim B\"uy\"ukboduk and Antonio Lei for enlightening correspondence on some of the topics of this paper. We also wish to thank the anonymous referee for helpful remarks and suggestions.

\section{Modular forms and Galois representations}

The goal of this section is to recall basic facts and fix some notation about modular forms and Galois representations associated with them.

\subsection{Galois representations of modular forms}\label{Galois_repr}

We briefly review Galois representations attached to modular forms. Since this will be the only interesting case for us in this paper, we consider modular forms with trivial character exclusively. 

\subsubsection{Notation for modular forms} \label{notation-forms-subsubsec}

From here on, let $k\geq2$ be an even integer, let $N\geq3$ be an integer and let $p\geq3$ be a prime number such that $p\nmid N$. Furthermore, let $\Gamma\in\bigl\{\Gamma_0(N),\Gamma_1(N)\cap\Gamma_0(p)\bigr\}$ and let $g\in S_k(\Gamma)$ be a normalized eigenform, whose $q$-expansion will be denoted by $g(q)=\sum_{n\geq1}a_n(g)q^n$. Write $\Q_g\defeq\Q\bigl(a_n(g)\mid n\geq1\bigr)$ for the Hecke field of $g$, \emph{i.e.}, the number field generated over $\Q$ by the Fourier coefficients of $g$, and $\cO_{\Q_g}$ for the ring of integers of $\Q_g$. The ring $\cO_g\defeq\Z\bigl[a_n(g)\mid n\geq1\bigr]$ is an order in $\cO_{\Q_g}$.

If $\Gamma=\Gamma_0(N)$ and $k=2$, then the Eichler--Shimura construction (\emph{cf.} \cite[\S 7.5]{shimura}) attaches to $g$ an abelian variety $A_g$ defined over $\Q$ whose dimension is $[\Q_g:\Q]$.  It is well known that $\cO_g$ embeds into the ring $\End_\Q(A_g)$ of the endomorphisms of $A_g$ defined over $\Q$; in fact, $\Q_g\simeq\End_\Q(A_g)\otimes_\Z\Q$, so that $A_g$ is of $\GL_2$-type (see, \emph{e.g.}, \cite[Corollary 4.2]{ribet-twists}). 

\begin{remark}
If $N$ is square-free (a condition that will be assumed later in this article), then $A_g$ is semistable, so all its endomorphisms are defined over $\Q$ (\cite[Corollary 1.4, (a)]{ribet-annals}). In particular, the ring of all endomorphisms of $A_g$ is commutative.
\end{remark}

\subsubsection{Galois representations} \label{representations-subsubsec}

We give a utilitarian overview of Galois representations that are associated with modular forms. We remark that, for the convenience of the reader, we adopt notation that is somewhat in line with the one used in \cite{BL}, \cite{KLZ} and \cite{LZ}, which are our main references for (big) Galois representations of (families of) modular forms.

Let $g$ be as in \S \ref{notation-forms-subsubsec}. To fix ideas, we treat the case $\Gamma=\Gamma_1(N)\cap\Gamma_0(p)$; see, \emph{e.g.}, \cite[\S 3.4]{LZ}, \cite[\S 2.1]{Vigni} for details on eigenforms of level $\Gamma_1(M)$ (\emph{i.e.}, eigenforms of level $\Gamma_0(M)$ and arbitrary character). Let $Y=Y(N,p)$ be the modular curve over $\Z[1/Np]$ classifying elliptic curves with $\Gamma_1(N)\cap\Gamma_0(p)$-level structure, denote by $Y_{\bar\Q}$ the base change of $Y$ to $\bar\Q$ and let $\pi\colon \mathcal{E}\rightarrow Y$ be the universal elliptic curve over $Y$. Set
\[ \mathscr{H}\defeq\mathbf{R}^1\pi_*\Z_p(1)=(\mathbf{R}^1\pi_*\Z_p)^\vee, \]
where $\star(1)$ stands for the Tate twist; this is a $\Z_p$-sheaf of rank $2$ on $Y$, whose associated $\Q_p$-sheaf will be denoted by $\mathscr H_{\Q_p}$. We write $\mathscr H_{\Q_p}^\vee$ for the dual of $\mathscr H_{\Q_p}$. Now let $\p$ be a prime of $\Q_g$ above $p$ and write $\Q_{g,\p}$ for the completion of $\Q_g$ at $\p$; by a slight abuse of notation, we write $\p$ also for the prime (actually, maximal) ideal of $\cO_g$ under $\p$. The $\p$-adic representation of $G_\Q$ attached to $g$ by Deligne (\cite{Del-Bourbaki}) can be defined as the \emph{maximal subspace} of the compactly supported cohomology
\[ H^1_{\text{\'et},c}\bigl(Y_{\bar\Q},\mathrm{Sym}^{k-2}(\mathscr{H}_{\Q_p}^\vee)\bigr)\otimes_{\Q_p}\Q_{g,\p} \]
on which the Hecke operators $T_\ell$ for all primes $\ell\nmid Np$ and $U_\ell$ for all primes $\ell\,|\,Np$ act as multiplication by the Fourier coefficient $a_\ell(g)$. It is a $2$-dimensional $\Q_{f,\p}$-vector space that will be denoted by $V_{g,\p}$ and is equipped with a continuous action of $G_\Q$ unramified outside the primes dividing $Np$. Let us write
\[ \rho_{g,\p}\colon G_\Q\longrightarrow \GL(V_{g,\p})\simeq \GL_2(\Q_{g,\p}) \]
for the associated group homomorphism. The representation $V_{g,\p}$ is characterized by the two conditions
\[ \tr\bigl(\rho_{g,\p}(\Frob_\ell)\bigr)=a_\ell(g),\quad\det\bigl(\rho_{g,\p}(\Frob_\ell)\bigr)=\ell^{k-1} \]
for all primes $\ell\nmid Np$, where $\Frob_\ell\in G_\Q$ is a \emph{geometric} Frobenius at $\ell$. 

The Galois representation $V_{g,\p}$ can be obtained as the $\p$-component of the $p$-adic \'etale realization of the motive attached to $g$ by Scholl (\cite{Scholl}), which lives in the \'etale cohomology of a certain compactification of the $(k-1)$-dimensional Kuga--Sato variety over $Y$. In this perspective, if $g$ has trivial character, then Poincar\'e duality on this Kuga--Sato variety gives rise to a $G_\Q$-equivariant, non-degenerate, alternating pairing
\begin{equation}\label{pairing}
V_{g,\p}\times V_{g,\p}\longrightarrow \Q_{g,\p}(1-k).
\end{equation}
This pairing induces a $G_\Q$-equivariant isomorphism
\begin{equation} \label{V-dual-eq}
V_{g,\p}^*\simeq V_{g,\p}(k-1)
\end{equation}
and implies that $V_{g,\p}^\dag\defeq V_{g,\p}(k/2)$ is the self-dual twist of $V_{g,\p}$, in the sense that $V_{g,\p}^\dag\simeq V_{g,\p}^{\dag,*}(1)$ where $V_{g,\p}^{\dag,*}$ is the dual (\emph{i.e.}, contragredient) representation of $V^\dagger_{g,\p}$. 

Let $\cO_{\Q_g,\p}$ be the valuation ring of $\Q_{g,\p}$. We write $T_{g,\p}$ for the $G_\Q$-stable $\cO_{\Q_g,\p}$-lattice in $V_{g,\p}$ generated by the image of the integral \'etale cohomology group
\[ H^1_{\text{\'et}}\bigl(Y_{\bar\Q},\mathrm{Sym}^{k-2}(\mathscr{H}^\vee)\bigr)\otimes_{\Z_p}\cO_{\Q_g,\p}, \]
where $\mathscr H^\vee$ is the dual of $\mathscr H$, and $T_{g,\p}^\dagger\defeq T_{g,\p}(k/2)$ for its self-dual twist, which is a $G_\Q$-stable $\cO_{\Q_g,\p}$-lattice in $V_{g,\p}^\dagger$. It turns out that the integral representation $\rho_{g,\p}:G_\Q\rightarrow\GL(T_{g,\p})\simeq\GL_2(\cO_{\Q_g,\p})$ is the $\p$-component of an integral $p$-adic representation 
\[ \rho_{g,p}:G_\Q\longrightarrow\GL_2(\cO_{\Q_g}\otimes_\Z\Z_p)\simeq\bigoplus_{\pi\mid p}\GL_2(\cO_{\Q_g,\pi}), \]
where $\pi$ varies over all the primes of $\Q_g$ above $p$ (see, \emph{e.g.}, \cite{ribet2}).

Let $\mathrm{TSym}^{k-2}(\mathscr{H}_{\Q_p})$ denote the sheaf of degree $k-2$ symmetric tensors over $\mathscr{H}_{\Q_p}$ (\emph{cf.} \cite[\S 2.2]{KLZ}). For reasons that will become clear later, we are also interested in the dual $V_{g,\p}^*$ of $V_{g,\p}$, which is the representation of $G_\Q$ that can be realized as the \emph{maximal quotient} of the cohomology
\[ H^1_{\text{\'et}}\bigl(Y_{\bar\Q},\mathrm{TSym}^{k-2}(\mathscr{H}_{\Q_p})(1)\bigr)\otimes_{\Q_p}\Q_{g,\p} \]
on which the action of the relevant Hecke operators at primes $\ell$ coincides with multiplication by the Fourier coefficient $a_\ell(g)$. Let us write
\[ \rho^*_{g,\p}\colon G_\Q\longrightarrow \GL(V^*_{g,\p})\simeq \GL_2(\Q_{g,\p}) \]
for the associated homomorphism. The representation $V_{f,\p}^*$ is characterized by the conditions
 \[ \tr\bigl(\rho^*_{g,\p}(\Frob_\ell^{-1})\bigr)=a_\ell(g),\quad\det\bigl(\rho^*_{g,\p}(\Frob_\ell^{-1})\bigr)=\ell^{k-1} \]
for all primes $\ell\nmid Np$, where $\Frob_\ell^{-1}\in G_\Q$ is an \emph{arithmetic} Frobenius at $\ell$. It follows from \eqref{pairing} that the self-dual twist of $V_{g,\p}^*$ is $V_{g,\p}^{*,\dag}\defeq V_{g,\p}^*(1-k/2)$, and then \eqref{V-dual-eq} implies that there is a canonical isomorphism
\begin{equation}\label{twist}
V_{g,\p}^{*,\dag}\overset\simeq\longrightarrow V_{g,\p}^\dag 
\end{equation}
of $G_\Q$-representations. In line with the notation for our distinguished lattice inside $V_{g,\p}$, we write $T_{g,\p}^*$ for the $G_\Q$-stable $\cO_{g,\p}$-lattice in $V_{g,\p}^*$ generated by the image of the integral \'etale cohomology
\[ H^1_{\text{\'et}}\bigl(Y_{\bar\Q},\mathrm{Tsym}^{k-2}(\mathscr{H})(1)\bigr)\otimes_{\Z_p}\cO_{\Q_g,\p} \]
and $T_{g,\p}^{*,\dag}\defeq T_{g,\p}^*(1-k/2)$ for its self-dual twist, which is a $G_\Q$-stable $\cO_{\Q_g,\p}$-lattice in $V_{g,\p}^{*,\dagger}$.

When $k=2$ and $\Gamma=\Gamma_0(N)$, there is a canonical $G_\Q$-equivariant isomorphism $V_{g,\p}\simeq V_\p(A_g)^*$, where $A_g$ is the abelian variety from \S \ref{notation-forms-subsubsec} and $V_\p(A_g)\defeq\Ta_\p(A_g)\otimes_{\cO_{\Q_g,\p}}\Q_{g,\p}$ is the $\Q_{g,\p}$-linear representation of $G_\Q$ associated with the $\p$-adic Tate module $\Ta_\p(A_g)$ of $A_g$ (the isomorphism above follows by combining \cite[Theorem 15.1, (a)]{Milne-AV} with the $G_\Q$-equivariant splitting $\Ta_p(A_g)=\oplus_{\pi\mid p}\Ta_\pi(A_g)$, where $\pi$ varies over all the primes of $\Q_g$ above $p$). Taking the self-dual twist, we obtain
\begin{equation} \label{V-2-eq}
V_{g,\p}^\dagger=V_{g,\p}(1)\simeq V_\p(A_g)^*(1)\simeq V_\p(A_g), 
\end{equation}
where the rightmost isomorphism, which we fix once and for all, is a consequence of the Weil pairing. Notice that the lattice $T_{g,\p}^\dagger$ corresponds to $\Ta_\p(A_g)$ under isomorphism \eqref{V-2-eq}. As for the dual, there is a canonical $G_\Q$-equivariant isomorphism $V_{g,\p}^*\simeq V_\p(A_g)$. Furthermore, $V_{g,\p}^{*,\dagger}=V_{g,\p}^*$, in accord with the fact that $V_\p(A_g)$ is self-dual.

\subsubsection{Residual Galois representations} \label{residual-subsubsec-1}

Define
\[ \bar{T}_{g,\p}\defeq T_{g,\p}\big/ \p T_{g,\p} \]
and let $\bar\rho_{g,\p}\colon G_\Q\rightarrow\GL(\bar{T}_{g,\p})\simeq\GL_2(\F_\p)$ be the associated homomorphism induced by $\rho_{g,\p}$, where $\F_\p\defeq\cO_{\Q_g,\p}/\p\cO_{\Q_g,\p}$. The representation $\bar\rho_{g,\p}$ (or, rather, its semisimplification) is the residual representation of $g$ at $\p$. Similarly, one can introduce representations $\bar T_{g,\p}^\dag$, $\bar T_{g,\p}^*$, $\bar T_{g,\p}^{*,\dagger}$ over $\F_\p$, whose corresponding homomorphisms will be denoted in the obvious fashion. It is essentially a consequence of Nakayama's lemma that if $\bar\rho_{f,\p}^\dag$ (or, equivalently, $\bar\rho_{f,\p}^{*,\dag}$) is (absolutely) irreducible (a condition that will be imposed in due course, \emph{cf.} \S \ref{choice-subsubsec}), then all $G_\Q$-stable $\cO_{\Q_g,\p}$-lattices inside $V_{g,\p}^\dagger$ are homothetic (see, \emph{e.g.}, \cite[Lemma 2.1]{vatsal-quebec}). It follows that, up to rescaling by an element of $\Q_{g,\p}^\times$ (an operation that induces isomorphisms in Galois cohomology), it is not restrictive to assume that isomorphism \eqref{twist} induces a $G_\Q$-equivariant isomorphism
\begin{equation} \label{Poinc}
T_{g,\p}^{*,\dagger}\overset\simeq\longrightarrow T_{g,\p}^\dag.
\end{equation}
In turn, isomorphism \eqref{Poinc} yields an isomorphism $\bar{T}_{g,\p}^{*,\dagger}\simeq\bar{T}_{g,\p}^\dagger$ of residual representations. 

Following \cite{LZ} (\emph{cf.} also \cite{BL}), in \S \ref{big-galois-subsec} we will see how to interpolate $V_{g,\p}^{*,\dagger}$ and $T_{g,\p}^{*,\dagger}$ as the eigenform $g$ varies in a $p$-adic family in the sense of Coleman (\cite{Coleman}).

\subsection{Algebraic ranks and Shafarevich--Tate groups of newforms} \label{sha-subsubsec}

Let us recall the newform $g\in S_k^\new(\Gamma_0(N))$ from \S \ref{notation-forms-subsubsec}. We choose a prime number $p\nmid N$, a prime $\p$ of $\Q_g$ above $p$ and introduce algebraic $\p$-ranks and $\p$-primary Shafarevich--Tate groups of $g$. 

\subsubsection{\'Etale Abel--Jacobi maps} \label{AJ-subsubsec}

Suppose that $k\geq4$. Let $\tilde\E_N^{k-2}$ be the Kuga--Sato variety of level $N$ and weight $k$ (see, \emph{e.g.}, \cite[\S 2.2]{LV}) and for any number field $F$ write $\CH^{k/2}\bigl(\tilde\E_N^{k-2}/F\bigr)_0$ for the Chow group of rational equivalence classes of codimension $k/2$ cycles on $\tilde\E_N^{k-2}$ defined over $F$ that are homologically trivial, \emph{i.e.}, belong to the kernel of the cycle class map in $\ell$-adic cohomology (see, \emph{e.g.}, \cite[\S 3.3]{Andre} or \cite[Chapter VI, \S 9]{Milne}; \emph{cf.} also \cite[\S 1.4]{Nek3} for details on the independence of the kernel above of the prime number $\ell$).

Let $g\in S_k^\new(\Gamma_0(N))$ be the newform that was fixed above. As explained, \emph{e.g.}, in \cite[\S 2.3]{LV}, there is an $\cO_{\Q_g,\p}$-linear \'etale Abel--Jacobi map 
\[ \Phi_{g,F,\p}:{\CH}^{k/2}\bigl(\tilde\E_N^{k-2}/F\bigr)_0\otimes_\Z\cO_{\Q_g,\p}\longrightarrow H^1\bigl(F,T_{g,\p}^\dagger\bigr). \]
From here on we set 
\begin{equation} \label{AJ-image-eq}
\Lambda_{g,\p}(F)\defeq\im(\Phi_{g,F,\p})\subset H^1\bigl(F,T_{g,\p}^\dagger\bigr).
\end{equation}
Now let $H^1_f\bigl(F,T^\dagger_{g,\p}\bigr)$ be the Bloch--Kato Selmer group of $T^\dagger_{g,\p}$ over $F$ (\cite{BK}; \emph{cf.} also \cite[\S 2.4]{LV}), which is a finitely generated $\cO_{g,\p}$-submodule of $H^1\bigl(F,T^\dagger_{g,\p}\bigr)$. Thanks to a result of Nizio\l, Nekov\'a\v{r} and Saito (see, \emph{e.g.}, \cite[Corollary 2.7, (2)]{LV}), there is an inclusion of $\cO_{\Q_g,\p}$-modules $\Lambda_{g,\p}(F)\subset H^1_f\bigl(F,T^\dagger_{g,\p}\bigr)$, so $\Lambda_{g,\p}(F)$ too is finitely generated over $\cO_{\Q_g,\p}$. 

\subsubsection{Algebraic $\p$-ranks}

We define the algebraic $\p$-ranks of our newform $g$ in terms of the \'etale Abel--Jacobi images introduced in \eqref{AJ-image-eq}. Let $F$ be a number field.

\begin{definition} \label{algebraic-rank-def}
The \emph{algebraic $\p$-rank of $g$ over $F$} is $r_{\alg,\p}(g/F)\defeq\rank_{\cO_{\Q_g,\p}}\Lambda_{g,\p}(F)\in\N$.
\end{definition}

This definition makes sense because, as we remarked in \S \ref{AJ-subsubsec}, the $\cO_{\Q_g,\p}$-module $\Lambda_{g,\p}(F)$ is finitely generated.

Now suppose that $g\in S_2^\new(\Gamma_0(N))$. The following is the counterpart in weight $2$ of Definition \ref{algebraic-rank-def}. 

\begin{definition} \label{algebraic-rank-def2}
The \emph{algebraic $\p$-rank of $g$ over $F$} is $r_{\alg,\p}(g/F)\defeq\rank_{\cO_{g,\p}}\bigl(A_g(F)\otimes_{\cO_g}\cO_{g,\p}\bigr)$.
\end{definition}

In particular, if $A_g$ is an elliptic curve (\emph{i.e.}, $a_n(g)\in\Q$ for all $n\geq1$) then $r_{\alg,p}(g/F)=\rank_{\Z_p}\bigl(A_g(F)\otimes_\Z\Z_p\bigr)=\rank_\Z A_g(F)$.

\subsubsection{Selmer and Shafarevich--Tate groups} \label{sha-subsubsec}

Again, let $k\geq4$. Let us consider the quotient $W^\dagger_{g,\p}\defeq V^\dagger_{g,\p}\big/T^\dagger_{g,\p}$. For any number field $F$, denote by $\Sel_\p(g/F)$ the \emph{$\p$-primary Selmer group of $g$ over $F$} as defined, in terms of $W_{g,\p}^\dagger$, in \cite[Definition 2.6]{LV}. Essentially as a consequence of work of Saito on the weight-monodromy conjecture for compactified Kuga--Sato varieties (\cite{saito}, \cite{saito2}) and of results of Nekov\'a\v{r} (\cite{Nek3}) and Nizio\l\ (\cite{niziol}) on $p$-adic regulators, there is an injection 
\[ i_{g,F,\p}:\Lambda_{g,\p}(F)\otimes_{\Z_p}\Q_p/\Z_p\longmono\Sel_\p(g/F) \]
of $\cO_{\Q_g,\p}$-modules (see, \emph{e.g.}, \cite[\S 2.4]{LV} for details).

The definition of Shafarevich--Tate group that we give below follows \cite{Nek}.

\begin{definition} \label{sha-def}
The \emph{$\p$-primary Shafarevich--Tate group of $g$ over $F$} is the quotient
\[ \Sha_\p(g/F)\defeq\Sel_\p(g/F)\big/i_{g,F,\p}\bigl(\Lambda_{g,\p}(F)\otimes_{\Z_p}\Q_p/\Z_p\bigr). \]
\end{definition}
Therefore, there is a (tautological) short exact sequence of $\cO_{\Q_g,\p}$-modules
\[ 0\longrightarrow\Lambda_{g,\p}(F)\otimes_{\Z_p}\Q_p/\Z_p\xrightarrow{i_{g,F,\p}}\Sel_\p(g/F)
\longrightarrow\Sha_\p(g/F)\longrightarrow0. \]
Theorem A in the introduction gives a result on algebraic $\p$-ranks and Shafarevich--Tate groups of $g$ when a $p$-stabilization of $g$ (see, \emph{e.g.}, \cite[\S 1.3]{bellaiche}) is a suitable specialization of a Coleman family of modular forms.

\begin{remark}
A different notion of Shafarevich--Tate group was proposed by Bloch and Kato in \cite{BK} (\emph{cf.} also \cite{Flach}). The $\p$-primary Shafarevich--Tate group of $g$ over $F$ in the sense of Bloch--Kato, which is defined as the quotient of $\Sel_\p(g/F)$ by its maximal divisible $\cO_{\Q_g,\p}$-submodule, is a finite quotient of the $\cO_{\Q_g,\p}$-module introduced in Definition \ref{sha-def}.
\end{remark}

\section{Coleman families of modular forms} \label{coleman-sec}

We introduce $p$-adic families of eigenforms \emph{\`a la} Coleman and their associated big Galois representations. In particular, we see how to interpolate the Galois representations from \S \ref{representations-subsubsec}.

%\subsection{On the Galois image} \label{image-subsec}

\subsection{Coleman families} \label{coleman-subsec}

Let $p\geq3$ be a prime number and let $N\geq1$ be an integer such that $p\nmid N$. We first review the notion of ``weight space'' as introduced by Coleman--Mazur in their paper on the eigencurve (\cite{ColMaz}), then turn to Coleman families of modular forms. 

\subsubsection{The weight space}

Set $\Z_{p,N}\defeq\varprojlim_r\Z/Np^r\Z=\Z/N\Z\times\Z_p$, so that $\Z_{p,N}^\times=(\Z/N\Z)^\times\times\Z^\times_p$. Furthermore, let us consider the Iwasawa algebra
\[ \Lambda_N\defeq\Z_p[[\Z_{p,N}^\times]]=\varprojlim_r\Z_p[(\Z/Np^r\Z)^\times]=\Z_p[[(\Z/N\Z)^\times\times\Z^\times_p]]. \]
Following Coleman--Mazur (\cite[\S 1.4]{ColMaz}), we define the \emph{$p$-adic weight space of tame level $N$} (or simply \emph{weight space}, if $N$ and $p$ are clear from the context) to be the rigid analytic space $\mathscr W_N$ over $\Q_p$ associated with the formal $\mathrm{Spf}(\Z_p)$-scheme $\mathrm{Spf}(\Lambda_N)$ (see, \emph{e.g.}, \cite[\S 1.1]{ColMaz} for details on this construction). For any complete field extension $\KK/\Q_p$ inside $\C_p$ there is a canonical identification
\[ \mathscr W_N(\KK)=\Hom_\cont\bigl(\Z_{p,N}^\times,\KK^\times\bigr) \]
between the set of $\KK$-valued points of $\mathscr W_N$ and the set of continuous characters $\Z_{p,N}^\times\rightarrow\KK^\times$. With standard notation, $\mathscr W_N$ represents the functor that takes a rigid analytic space $X$ over $\Q_p$ to the set $\Hom_\cont\bigl(\Z_{p,N}^\times,\cO_X(X)^\times\bigr)$. Let $\phi$ be the classical Euler function; somewhat more concretely, $\mathscr W_N$ can be described as the disjoint union of $\phi(Np)$ copies of the open unit disc about $1\in\C_p^\times$, these discs being indexed by the characters $(\Z/Np\Z)^\times\rightarrow\C_p^\times$. 

There is a canonical embedding $\Z\hookrightarrow\mathscr W_N(\Q_p)$ that sends $k\in\Z$ to the composition of the map $\Z_{p,N}^\times\rightarrow\Z_{p,N}^\times$, $x\mapsto x^{k-2}$ with the projection $\Z_{p,N}^\times\twoheadrightarrow\Z_p^\times$ and the natural injection $\Z_p^\times\hookrightarrow\Q_p^\times$; in the rest of the paper, we view this embedding as an inclusion $\Z\subset\mathscr W_N(\Q_p)$. A point of $\mathscr W_N(\Q_p)$ corresponding to an element of $\Z_{\geq2}$ under the inclusion above will be called a \emph{classical point} of $\mathscr W_N$. It is convenient to introduce the set
\begin{equation} \label{classical-weights-eq}
\mathscr W_{N,\cla}\defeq\bigl\{\text{classical points of $\mathscr W_N(\Q_p)$}\bigr\}; 
\end{equation}
we shall call it the set of \emph{classical weights} of tame level $N$. Equivalently, in light of the inclusion $\Z\subset\mathscr W_N(\Q_p)$, we can write $\mathscr W_{N,\cla}\defeq\Z_{\geq2}\subset\mathscr W_N(\Q_p)$. Given $n\in\Z$, we also define
\begin{equation} \label{classical-weights-n-eq}
\mathscr W_{N,\cla,n}\defeq\bigl\{k\in\mathscr W_{N,\cla}\mid k\geq n\bigr\}.
\end{equation}
Finally, a point of $\mathscr W_N(\Q_p)$ corresponding to an even integer in $\Z_{\geq2}$ is an \emph{even classical point} of $\mathscr W_N$. We define
\begin{equation} \label{even-classical-weights-eq}
\mathscr W_{N,\cla}^0\defeq\bigl\{\text{even classical points of $\mathscr W_N(\Q_p)$}\bigr\}
\end{equation}
and call $\mathscr W_{N,\cla}^0$ the set of \emph{even classical weights} of tame level $N$.

\begin{remark} \label{weight-inclusion-rem}
If $N$ divides $M$, then the canonical projection $\Z_{p,M}\twoheadrightarrow\Z_{p,N}$ induces an injection $\mathscr W_N\hookrightarrow\mathscr W_M$ that allows one to identify $\mathscr W_N$ with an open and closed rigid analytic subgroup of $\mathscr W_M$.
\end{remark}

\begin{remark} \label{restrictive-rem}
From the point of view of the eigencurve (\cite{Buzzard-eigenvarieties}, \cite{ColMaz}), the definition of classical points (and classical weights) given above is rather restrictive: see, \emph{e.g.}, \cite[p. 34]{ColMaz} and \cite[Remark 4.6.2]{LZ} for details. However, in this article we will not be interested in modular forms with non-trivial character, so this definition perfectly fits our needs.
\end{remark}

\subsubsection{Coleman families} \label{coleman-subsubsec}

We briefly review definitions and basic properties of $p$-adic families of modular forms in the sense of Coleman; here we follow \cite[\S 4.6]{LZ} quite closely. Let us fix a finite extension $L$ of $\Q_p$, which will play the role of a coefficient field, denote by $\cO_L$ the valuation ring of $L$ and take a wide open disc $U\subset\mathscr W_N$ defined over $L$ such that the classical weights $U\cap\mathscr W_{N,\cla}$, with $\mathscr W_{N,\cla}$ as in \eqref{classical-weights-eq}, are dense in $U$. Furthermore, write $\Lambda_U$ for the $\cO_L$-algebra of rigid functions on $U$ bounded by $1$; there is an isomorphism $\Lambda_U\simeq\cO_L[[T]]$, from which it follows that $\Lambda_U$ is a 2-dimensional complete noetherian regular local ring (in particular, $\Lambda_U$ is a unique factorization domain). Finally, set $\mathcal B_U\defeq\Lambda_U[1/p]$.

\begin{definition} \label{coleman-def}
A ($p$-adic) \emph{Coleman family} over $U$ (of tame level $N$) is a formal $q$-expansion
\[ \f_U=\sum_{n\geq1}a_n(\f_U)q^n\in\Lambda_U[[q]] \]
with $a_1(\f_U)=1$ and $a_p(\f_U)$ invertible in $\mathcal B_U$ such that for all but finitely many $k\in U\cap\mathscr W_{N,\cla}$ the power series
\begin{equation} \label{specialization-k-eq}
f_k\defeq\f_U(k)=\sum_{n\geq1}a_n(\f_U)(k)q^n\in\cO_L[[q]] 
\end{equation}
is the $q$-expansion of a classical (normalized) eigenform of weight $k$ and level $\Gamma_1(N)\cap\Gamma_0(p)$.
\end{definition}

The cusp form $f_k$ in \eqref{specialization-k-eq} is the \emph{specialization of $\f_U$ at $k$}. In this case, with non-standard (but useful) terminology, we say that $k\in U\cap\mathscr W_{N,\cla}$ is \emph{admissible}; see Remark \ref{slope-rem2} for a fundamental result on admissible weights.

\begin{remark}
Unlike Hida families, which are defined over the whole $\mathscr W_N$, Coleman families exist only locally over $\mathscr W_N$, as is apparent from Definition \ref{coleman-def}.
\end{remark}

\begin{remark}
In complete analogy with what we observed in Remark \ref{restrictive-rem} for classical weights, Definition \ref{coleman-def} could be made more general: \emph{cf.} \cite[Remark 4.6.2]{LZ}.
\end{remark}

\begin{remark} \label{fourier-rem}
As explained in \cite[\S 3.1]{Vigni}, for every admissible $k\in U\cap\mathscr W_{N,\cla}$ we may (and do) view $f_k$ as a classical modular form with complex Fourier coefficients.
\end{remark}

We assume throughout that $f_k$ is $N$-new for all admissible $k\in U\cap\mathscr W_{N,\cla}$.  It follows that for every $k$ the form $f_k$ is a $p$-stabilized newform, \emph{i.e.}, either $f_k$ is $p$-new too or there is a normalized newform $f_k^\flat\in S_k^\new(\Gamma_1(N))$ such that there is an equality of $q$-expansions
\[ f_k(q)=f_k^\flat(q)-\frac{p^{k-1}}{a_p(\f_U)(k)}f_k^\flat(q^p). \]
As a consequence of the strong multiplicity one theorem (see, \emph{e.g.}, \cite[Theorem 3.22]{hida-modular}), a normalized newform $f_k^\flat$ as above is unique. When $f_k$ is also $p$-new (and so new of level $\Gamma_1(N)\cap\Gamma_0(p)$) we set $f_k^\flat\defeq f_k$. 

\begin{remark}
As was emphasized in Remark \ref{restrictive-rem}, in the present paper we will be interested exclusively in specializations with trivial character, so that $f_k^\flat\in S_k^\new(\Gamma_0(M))$ with $M=N$ or $M=Np$.
\end{remark}

\begin{notationt} \label{notation-rem}
We view the specializations $f_k$ as classical modular forms (\emph{cf.} Remark \ref{fourier-rem}), so the Hecke fields $\Q_{f_k^\flat}$ of the newforms $f_k^\flat$ are number fields, which we regard as subfields of $\bar\Q$ via the embeddings from \S \ref{notation-subsec}. From here on, we fix a prime ideal $\fP$ of $\bar\Z$ above $p$ and use it as a shorthand for $\fP\cap\Q_{f_k^\flat}$ whenever the form $f_k^\flat$ is involved (therefore, \emph{e.g.}, $\cO_{\Q_{f_k^\flat},\fP}$ denotes the completion of the ring of integers $\cO_{\Q_{f_k^\flat}}$ of $\Q_{f_k^\flat}$ at $\fP\cap\Q_{f_k^\flat}$). However, in order to further lighten our notation, we simply write $V_{f_k^\flat}$ in place of $V_{f_k^\flat,\fP}$. Analogously, $T_{f_k^\flat,\fP}$ will be denoted by $T_{f_k^\flat}$ and a similar convention will apply to residual representations. Terminology-wise, we use the expression ``algebraic $\fP$-rank of $f_k^\flat$'' to indicate the algebraic $\fP\cap\Q_{f_k^\flat}$-rank of $f_k^\flat$ in the sense of Definition \ref{algebraic-rank-def}.
\end{notationt}

%Finally, we simply write $\f$ in place of $\f_U$ when there is no risk of confusion.

\subsubsection{Interpolation in Coleman families}

Coleman's theory (\cite{Coleman}) deals with the question of interpolating modular eigenforms in $p$-adic families. To begin with, let us fix an extension $\ord_p:\bar\Q_p^\times\rightarrow\Q$ of the $p$-adic valuation on $\Q_p$ normalized by $\ord_p(p)=1$. Quite generally, let $g\in S_k(\Gamma_1(Np^r))$, with $r\geq1$, be a normalized eigenform for the Hecke operator $U_p$. With self-explaining notation, the \emph{slope} of $g$ is the $p$-adic valuation $\ord_p\bigl(a_p(g)\bigr)$ of the eigenvalue of $U_p$. It turns out that if $g$ is of level $\Gamma_1(N)\cap\Gamma_0(p)$, then
\begin{equation} \label{slope-ineq}
0\leq\ord_p\bigl(a_p(g)\bigr)\leq k-1
\end{equation}
when $g$ is old at $p$, while $\ord_p\bigl(a_p(g)\bigr)=(k-2)/2$ if $g$ is new at $p$. See, \emph{e.g.}, \cite{buzzard-slopes}, \cite{BG-slopes}, \cite{gouvea-slopes} for further details, results and conjectures about slopes of modular forms.

\begin{remark} \label{critical-rem}
If $\ord_p\bigl(a_p(g)\bigr)=k-1$, then $g$ is said to be \emph{of critical slope} (see, \emph{e.g.}, \cite[\S 2.2.2]{bellaiche} for details).
\end{remark}

Now set $G_{\Q_p}\defeq\Gal(\bar\Q_p/\Q_p)$; the embedding $\iota_p$ from \S \ref{notation-subsec} induces an inclusion $G_{\Q_p}\subset G_\Q$. Given an eigenform $g$ as above, fix a prime $\p$ of $\Q_g$ lying over $p$ and consider the Galois representation $\rho_{g,\p}:G_\Q\rightarrow\GL_2(\Q_{g,\p})$ attached to $g$ by Deligne (\cite{Del-Bourbaki}). As in \cite{LZ}, it is useful to introduce ``noble eigenforms'', whose notion (in a cohomological setting) is due to Hansen. In doing this, we will need $p$-stabilizations of newforms of level $\Gamma_1(N)$, which are closely related to the $p$-stabilized newforms from \S \ref{coleman-subsubsec}: see, \emph{e.g.}, \cite[\S 1.3]{bellaiche} and \S \ref{stabilizations-subsubsec} for definitions and details.

\begin{definition} \label{noble-def}
A ($p$-adic) \emph{noble eigenform} (of tame level $N$) is a normalized eigenform $g$ of weight $k\geq2$ and level $\Gamma_1(N)\cap\Gamma_0(p)$ such that:
\begin{enumerate}
\item $g$ is a $p$-stabilization of a normalized newform $g^\flat$ of weight $k$, level $N$ and character $\chi_{g^\flat}$ whose Hecke polynomial $X^2-a_p(g^\flat)X+\chi_{g^\flat}(p)p^{k-1}$ has distinct roots (``$p$-regularity'');
\item if $\ord_p\bigl(a_p(g)\bigr)=k-1$, then the local Galois representation ${\rho_{g^\flat,\p}|}_{G_{\Q_p}}$ is not the direct sum of two characters (``non-criticality'').
\end{enumerate}
\end{definition}

The following interpolation result, which shows the existence of the Coleman families we study in this paper, is a consequence of the fact that the eigencurve of tame level $N$ (\cite{Buzzard-eigenvarieties}, \cite{ColMaz}) is \'etale over the weight space $\mathscr W_N$ at points corresponding to noble eigenforms (see, \emph{e.g.}, \cite[Proposition 2.11]{bellaiche}).

\begin{theorem} \label{interpolation-thm}
Let $g$ be a noble eigenform of weight $k_0$ and tame level $N$. There exist an open disc $U\ni k_0$ in $\mathscr W_N$ and a unique Coleman family $\boldsymbol g_U$ over $U$ of tame level $N$ such that $g_{k_0}=g$.
\end{theorem}

\begin{proof} This is \cite[Theorem 4.6.4]{LZ}. \end{proof}

We say that $\boldsymbol g_U$ as in Theorem \ref{interpolation-thm} is the Coleman family over $U$ passing through $g$ (or, by a slight abuse of terminology, through $g^\flat$).

\begin{remark} \label{slope-rem}
Since the coefficients $a_n(\boldsymbol g_U)$ are rigid analytic, at the cost of shrinking $U$ one can always assume that the slope function $k\mapsto\ord_p\bigl(a_p(g_k)\bigr)$ is constant over the admissible weights in the domain (\emph{cf.} \S \ref{coleman-subsubsec} for the terminology). In the remainder of this article, we tacitly suppose that this condition is satisfied, as it is one of the running assumptions in papers whose results we shall use; however, such a constancy property will play no explicit role in our arguments.
\end{remark}

\begin{remark} \label{slope-rem2}
As announced in Remark \ref{slope-rem}, let us assume that the slope $\ord_p\bigl(a_p(g_k)\bigr)$ is constant for all admissible $k\in U\cap\mathscr W_{N,\cla}$; call $s$ this common value. It follows immediately from \eqref{slope-ineq} that $k\geq s+1$. Conversely, every $k\in U\cap\mathscr W_{N,\cla}$ with $k>s+1$ is admissible (\cite[Corollary B5.7.1]{Coleman}).
\end{remark}

\begin{remark}
One can show that if the slopes $\ord_p\bigl(a_p(g_k)\bigr)$ are constant (\emph{cf.} Remark \ref{slope-rem}), then there is at most one admissible $k_0\in U\cap\mathscr W_{N,\cla}$ such that $g_{k_0}$ is $p$-new: see, \emph{e.g.}, \cite[p. 169]{GSS} for details.
\end{remark}

\subsection{Big Galois representations} \label{big-galois-subsec}

Following \cite{LZ} (\emph{cf.} also \cite{BL}), we introduce the big Galois representation attached to a Coleman family. Here, with notation close to that of \cite{LZ}, we simply sketch the construction of this representation, referring to \cite[Section 4]{LZ} and \cite{BL} for details.

\subsubsection{Some function spaces}

Let $\Lambda_U$ be as in \S \ref{coleman-subsubsec}, let $\Lambda(\Z_p^\times)\defeq\Z_p[[\Z_p^\times]]$ be the Iwasawa algebra of $\Z_p^\times$ and let 
\begin{equation} \label{universal-eq}
\kappa_U:\Z_{p,N}^\times\longepi\Z_p^\times\longmono\Lambda(\Z_p^\times)^\times\longrightarrow\Lambda_U^\times 
\end{equation}
be the universal character, where $\Z_{p,N}^\times\twoheadrightarrow\Z_p^\times$ is the canonical projection. Let us define the subset $T\defeq p\Z_p\times \Z_p^\times$ of $\Z_p\times \Z_p$ and write $A_U^\circ(T)$ for the set of functions $h:T\rightarrow\Lambda_U$ such that
\begin{itemize}
\item $h(\gamma t)=\kappa_U(\gamma)h(t)$ for all $\gamma\in \Z_{p,N}^\times$;
\item $z\mapsto h(pz,1)$ is a locally analytic function on $\Z_p$.
\end{itemize} 
Of course, the multiplication $\gamma t$ is induced by $\Z_{p,N}^\times\twoheadrightarrow\Z_p^\times$. The abelian group $A_U^\circ(T)$ has an obvious $\Lambda_U$-module structure and we consider its dual $D_U^\circ(T)\defeq\Hom_{\Lambda_U}\bigl(A_U^\circ(T),\Lambda_U\bigr)$. Similarly, for every $k\in U\cap \mathscr W_{N,\cla}$ let $A_k^\circ(T)$ be the set of functions $h:T\rightarrow\cO_L$ such that 
\begin{itemize}
\item $h(\gamma z)=\gamma^{k-2}h(z)$ for all $\gamma \in \Z_{p,N}^\times$;
\item $z\mapsto h(pz,1)$ is a locally analytic function on $\Z_p$. 
\end{itemize}
Let $D_k^\circ(T)\defeq\Hom_{\cO_L}\bigl(A_k^\circ(T),\cO_L\bigr)$ be the $\cO_L$-linear dual of $A_k^\circ(T)$. The specialization (\emph{i.e.}, evaluation) map $\Lambda_U\rightarrow \cO_L$ induces a group homomorphism
\begin{equation} \label{D-eq}
D_U^\circ(T) \longrightarrow D_k^\circ(T). 
\end{equation}
Let $P_k^\circ$ be the space of polynomial functions with coefficients in $\cO_L$ on $\Z_p\times \Z_p$ that are homogeneous of degree $k-2$. Restriction to $T$ produces a natural embedding $P_k^\circ \hookrightarrow A_k^\circ(T)$, which in turn induces a canonical surjection 
\begin{equation} \label{onto-eq}
D_k^\circ(T) \longepi (P_k^\circ)^*=\mathrm{TSym}^{k-2}(\cO_L\times \cO_L).
\end{equation}
Composing \eqref{D-eq} and \eqref{onto-eq}, we obtain a map
\begin{equation} \label{spec}
D_U^\circ(T) \longrightarrow D_k^\circ(T)\longrightarrow \mathrm{TSym}^{k-2}(\cO_L\times \cO_L)
\end{equation}
that is Hecke-equivariant (see, \emph{e.g.}, \cite[\S 4.2]{LZ} for details).

\subsubsection{Some modular curves}

For every $n\in\N$ let $Y(p^n,Np^n)$ be the modular curve defined in \cite[\S 2.1]{Kato}; consider the pro-scheme $Y(p^\infty,N p^\infty)\defeq\varprojlim_n Y(p^n,N p^n)$. As in \S \ref{representations-subsubsec}, let $Y(N,p)$ be the modular curve over $\Z[1/Np]$ classifying elliptic curves with $\Gamma_0(N)\cap\Gamma_1(p)$-level structure. There is a natural projection 
\begin{equation} \label{Y-proj-eq}
Y(p^\infty,N p^\infty)\longrightarrow Y(N,p). 
\end{equation}
 As explained in \cite[Proposition 4.4.3]{LZ}, there exist a sheaf $\mathcal{D}_U^\circ$ of $\Lambda_U$-modules and a sheaf $\mathcal{D}_k^\circ$ of $\cO_L$-modules on $Y(N,p)$ whose pullbacks to $Y(p^\infty, N p^\infty)$ under \eqref{Y-proj-eq} are the constant sheaves $D_U^\circ(T)$ and $D_k^\circ(T)$, respectively. Furthermore, it turns out (\emph{cf.} \cite[Proposition 4.4.2]{LZ}) that the sequence in \eqref{spec} induces a sequence 
\begin{equation} \label{D-composition-eq}
\mathcal{D}_U^\circ \longrightarrow \mathcal{D}_k^\circ\longrightarrow \mathrm{TSym}^{k-2}(\mathscr{H}_{\cO_L}) 
\end{equation}
of pro-sheaves on $Y(N,p)$, which is Hecke-equivariant. To ease our notation, set $Y\defeq Y(N,p)$ and denote by $Y_{\bar\Q}$ the base change of $Y$ to $\bar\Q$. The (twisted) \'etale cohomology group $H^1_{\text{\'et}}\bigl(Y_{\bar\Q},\mathcal{D}_U^\circ\bigr)(1)$ is a $\Lambda_U$-module that is equipped with a natural continuous action of $G_\Q$ unramified outside $Np\infty$. Analogously, $H^1_{\text{\'et}}\bigl(Y_{\bar\Q},\mathcal{D}_k^\circ\bigr)(1)$ is an $\cO_L$-module with a continuous action of $G_\Q$ unramified outside $Np\infty$. The Hecke operators $T_n^\circ$ induced on $H^1_{\text{\'et}}\bigl(Y_{\bar\Q},\mathcal{D}_U^\circ\bigr)(1)$ and $H^1_{\text{\'et}}\bigl(Y_{\bar\Q},\mathcal{D}_k^\circ\bigr)(1)$ by the corresponding actions on sheaves commute with the Galois action (see, \emph{e.g.}, \cite[\S 3.3]{AIS}), so the map in \eqref{D-composition-eq} gives rise to a map
\begin{equation} \label{specialization}
H^1_{\text{\'et}}\bigl(Y_{\bar\Q},\mathcal{D}_U^\circ\bigr)(1)\longrightarrow H^1_{\text{\'et}}\bigl(Y_{\bar\Q},\mathcal{D}_k^\circ\bigr)(1)\longrightarrow H^1_{\text{\'et}}\bigl(Y_{\bar\Q},\mathrm{TSym}^{k-2}(\mathscr{H}_{\cO_L})\bigr)(1)
\end{equation}
that is compatible with both the Galois action and the Hecke action.

\subsubsection{Big Galois representations}

Let $f$ be a noble eigenform of weight $2$ and tame level $N$ in the sense of Definition \ref{noble-def}. By Theorem \ref{interpolation-thm}, there exist an open disc $U\ni 2$ in $\mathscr W_N$ and a unique $p$-adic Coleman family $\f_U=\sum_{n\geq1}a_n(\f_U)q^n\in\Lambda_U[[q]]$ of tame level $N$ over $U$ passing through $f$. Set
\[ \V_{\f_U}\defeq\bigcap_{n\geq1}\Bigl(H^1_{\text{\'et}}\bigl(Y_{\bar\Q},\mathcal{D}_U^\circ\bigr)(1)[1/p]\Bigr)^{\!T_n^\circ=a_n(\f_U)}, \]
which inherits a continuous action of $G_\Q$ from that on $H^1_{\text{\'et}}\bigl(Y_{\bar\Q},\mathcal{D}_U^\circ\bigr)(1)$. For every admissible $k\in U\cap\mathscr W_{N,\cla}$, write $\wp_k$ for the prime ideal of $\Lambda_U$ corresponding to the character $z\mapsto z^{k-2}$. With $\mathcal B_U$ as in \S \ref{coleman-subsubsec}, it is proved in \cite[Theorem 4.6.6]{LZ} that, up to shrinking $U$ if necessary, $\V_{\f_U}$ is a free $\mathcal B_U$-module of rank $2$ such that for every $k$ as above there is an isomorphism 
\[ \V_{\f_U}\big/\wp_k\V_{\f_U}\overset\simeq\longrightarrow V_{f_k}^* \]
of $L$-linear representations of $G_\Q$. In the rest of this paper, we always assume that $U$ is small enough so that the above-mentioned freeness condition is satisfied. Let
\[ \rho_{\f_U}:G_\Q\longrightarrow\GL(\V_{\f_U})\simeq\GL_2(\mathcal B_U) \]
be the associated group homomorphism. By a standard density argument, the specialization property implies that
\begin{equation}\label{key_prop}
\tr\bigl(\rho_{\f_U}(\Frob_\ell^{-1})\bigr)=a_\ell(\f_U)
\end{equation}
for every prime number $\ell\nmid Np$, where $\Frob_\ell\in G_\Q$ is an arithmetic Frobenius at $\ell$.  It turns out (\cite[Theorem 4.6.6]{LZ}) that $\V_{\f_U}$ is a direct summand of $H^1_{\text{\'et}}\bigl(Y_{\bar\Q},\mathcal{D}_U^\circ\bigr)(1)[1/p]$, so there is a natural (projection) map
\begin{equation} \label{summand-eq}
H^1_{\text{\'et}}\bigl(Y_{\bar\Q},\mathcal{D}_U^\circ\bigr)(1)[1/p]\longepi\V_{\f_U}.
\end{equation}
As in \cite{BL}, we consider the reflexive hull, \emph{i.e.}, the double dual, $\V^\circ_{\f_U}$ of the image $\mathcal V_{\f_U}$ in $\V_{\f_U}$ of the $\Lambda_U$-module $H^1_{\text{\'et}}(Y_{\overline{\Q}},\mathcal{D}_U^\circ)(1)$ under the map induced by \eqref{summand-eq}. The $\Lambda_U$-module $\mathcal V_{\f_U}$ is finitely generated and, since $\Lambda_U$ is noetherian, the same is true of $\V^\circ_{\f_U}$. On the other hand, $\Lambda_U$ is a $2$-dimensional regular local ring, so the finitely generated, reflexive $\Lambda_U$-module $\V^\circ_{\f_U}$ is free (see, \emph{e.g.}, \cite[Proposition 5.1.9]{NSW}). It follows that $\V^\circ_{\f_U}$ is a free $\Lambda_U$-module of rank $2$ equipped with a continuous action of $G_\Q$ unramified outside $Np\infty$. Let us write
\begin{equation} \label{rho-circ-eq}
\rho^\circ_{\f_U}:G_\Q\longrightarrow\GL\bigl(\V^\circ_{\f_U}\bigr)\simeq\GL_2(\Lambda_U) 
\end{equation}
for the associated homomorphism. 

\begin{remark}
Since $\mathcal V_{\f_U}$ is torsion-free over $\Lambda_U$, there is a natural identification
\[ \V^\circ_{\f_U}=\bigcap_{\substack{\wp\in\Spec(\Lambda_U)\\[0.5mm]\mathrm{ht}(\wp)=1}}\mathcal V_{\f_U,\wp}, \]
where $\wp$ varies over all the height $1$ prime ideals of $\Lambda_U$ and $\mathcal V_{\f_U,\wp}$ is the localization of $\mathcal V_{\f_U}$ at $\wp$ (see, \emph{e.g.}, \cite[Lemma 5.1.2, (ii)]{NSW}). 
\end{remark}

There are $G_\Q$-equivariant specialization maps
\begin{equation}\label{spec_latt}
\V^\circ_{\f_U}\longrightarrow\V^\circ_{\f_U}\big/\wp_k\V^\circ_{\f_U}\longrightarrow T_{f_k}^* 
\end{equation}
induced by  \eqref{specialization}. Finally, recall the universal character $\kappa_U$ from \eqref{universal-eq} and set 
\[ \T_{\f_U}^\dagger\defeq p^{-\mathscr C}\V^\circ_{\f_U}\Bigl(1-\frac{\kappa_U}{2}\Bigr) \]
for a suitably large $\mathscr C\in\Z$ that we choose as explained in \cite[Remark A.1]{BL}. There is an associated homomorphism
\[ \rho^\dagger_{\f_U}:G_\Q\longrightarrow\GL\bigl(\T^\dagger_{\f_U}\bigr)\simeq\GL_2(\Lambda_U). \]
%Note that, for every $k\in U\cap \mathscr{W}_{N,\text{cl}}$ one has $(\T_{\f_U}^*)_{\wp_k}=V(\f_U)$ therefore the specialization maps gives rise to isomorphisms
%\[ (\T_{\f_U}^*)_{\wp_k}/\wp_k(\T_{\f_U}^*)_{\wp_k}\simeq V_{\f_k}^* \]
%which are exactly the isomorphisms in \eqref{iso}.
For every admissible $k\in U\cap\mathscr W^0_{N,\cla}$ there is a $G_\Q$-equivariant (specialization) map
\begin{equation} \label{sp}
\text{sp}_k:\T_{\f_U}^\dagger\longrightarrow T_{f_k}^*\Bigl(1-\frac{k}{2}\Bigr)=T_{f_k}^{*,\dag}.
\end{equation}
If $\T_{\f_U,\wp_k}^\dagger$ is the localization of $\T_{\f_U}^\dagger$ at $\wp_k$, then we get isomorphisms of $G_\Q$-representations
\[ \T_{\f_U,\wp_k}^\dagger\big/\wp_k\T_{\f_U,\wp_k}^\dagger\overset\simeq\longrightarrow V_{f_k}^{*,\dag}\overset\simeq\longrightarrow V_{f_k}^\dag, \] 
the former being induced by the map $\text{sp}_k$ from \eqref{sp}, the latter coming from an incarnation of Poincar\'e duality (\emph{cf.} \S \ref{representations-subsubsec}).

\subsubsection{Residual representations} \label{residual-subsubsec}

Denote by $\m_{\Lambda_U}$ the maximal ideal of $\Lambda_U$ and let $\F_{\Lambda_U}\defeq\Lambda_U/\m_{\Lambda_U}$ be the residue field of $\Lambda_U$, which we view as a subfield of $\bar\F_p$. Reduction modulo $\m_{\Lambda_U}$ yields a $2$-dimensional representation $\bar\V_{\f_U}^\circ\defeq\V_{\f_U}^\circ\big/\m_{\Lambda_U}\V_{\f_U}^\circ$ of $G_\Q$ over $\F_{\Lambda_U}$ that is unramified outside $pN\infty$. Let us write
\[ \bar{\rho}^\circ_{\f_U}\colon G_\Q\longrightarrow\GL\bigl(\bar\V_{\f_U}^\circ\bigr)\simeq\GL_2(\F_{\Lambda_U}) \]
for the induced homomorphism, which is just the reduction modulo $\m_{\Lambda_U}$ of the map $\rho^\circ_{\f_U}$ in \eqref{rho-circ-eq}. Let $\bar{a}_\ell(\f_U)\in\F_{\Lambda_U}$ stand for the class of $a_\ell(\f_U)$ modulo $\m_{\Lambda_U}$; as a consequence of \eqref{key_prop}, there is an equality 
\[ \tr\bigl(\bar\rho^\circ_{\f_U}(\Frob_\ell^{-1})\bigr)=\bar{a}_\ell(\f_U) \]
for every prime $\ell\nmid pN$. We deduce that
\[ \tr\bigl(\bar\rho^\circ_{\f_U}(\Frob_\ell^{-1})\bigr)=\tr\bigl(\bar\rho_{f_k}^*(\Frob_\ell^{-1})\bigr) \]
for all primes $\ell\nmid pN$ and all admissible $k\in U\cap\mathscr{W}^0_{N,\cla}$. It follows that the specialization map \eqref{spec_latt} induces an isomorphism
\[ \bar\V^\circ_{\f_U}\overset{\simeq}{\longrightarrow}\bar{T}_{f_k}^* \]
for every admissible $k\in U\cap\mathscr{W}^0_{N,\cla}$. In particular, $\bar{T}_f^*\simeq\bar{T}_{f_k}^*$ for every $k$ as above.

Now let $\varepsilon_{\text{cyc}}\colon G_\Q\rightarrow \Z_p^\times$ be the $p$-adic cyclotomic character and write $\bar{\varepsilon}_{\text{cyc}}\colon G_\Q\rightarrow \F_p^\times$ for the reduction of $\varepsilon_{\text{cyc}}$ modulo $p$. If $k\in U\cap\mathscr{W}_{N,\text{cl}}$ is such that $k\equiv 2 \pmod{2(p-1)}$, then $\bar{\varepsilon}_{\text{cyc}}^{1-k/2}$ is trivial, hence $\bar{T}_{f_k}^{*,\dag}=\bar{T}_{f_k}^*$. It follows that for any such weight $k$ the specialization maps in \eqref{sp} give rise to a commutative diagram
\begin{equation} \label{diagram}
\xymatrix@C=38pt@R=25pt{
T_f^\dag\ar[d]&T_f^{*,\dag}\ar[d]\ar[l]_-\simeq & \T_{\f_U}^\dag \ar[l]_-{\mathrm{sp}_2}\ar[r]^-{\mathrm{sp}_k} & T_{f_k}^{*,\dag}\ar[d]\ar[r]^-\simeq & T_{f_k}^\dag\ar[d]\\
\bar{T}_f^\dag& \bar{T}_f^*\ar[l]_-\simeq &  & \bar{T}_{f_k}^*\ar[ll]_-\simeq & \bar{T}_{f_k}^\dag \ar[l]_-\simeq
}
\end{equation}
where $\mathrm{sp}_2$ and $\mathrm{sp}_k$ are the maps from \eqref{sp}, the vertical arrows are induced by reduction modulo the maximal ideal of $\cO_L$ and the isomorphisms come from \eqref{Poinc}.

\section{Root numbers and Greenberg's conjecture} \label{greenberg-sec}

We recall properties of root numbers in Coleman families and then state a special case of
Greenberg's conjecture on analytic ranks of families of eigenforms.

\subsection{Root numbers in Coleman families} 

We review results on the variation of root numbers in (and attach a root number to) a Coleman family of modular forms.

\subsubsection{Root numbers and analytic ranks of newforms}

Let $g\in S_k(\Gamma_0(M))$ be a normalized newform of weight $k\geq2$ and write $L(g,s)$ for its (complex) $L$-function. The completed $L$-function of $g$ is $\Lambda(g,s)\defeq(2\pi)^{-s}\Gamma(s)M^{s/2}L(g,s)$, where $\Gamma(s)$ is the classical $\Gamma$-function. It is well known that $\Lambda(g,s)$ satisfies a functional equation 
\begin{equation} \label{L-functional-eq}
\Lambda(g,s)=\varepsilon(g)\Lambda(g,k-s) 
\end{equation}
where $\varepsilon(g)\in\{\pm1\}$ is the \emph{root number} of $g$ (see, \emph{e.g.}, \cite[Theorem 9.27]{Knapp}). 

\begin{definition} \label{analytic-rank-def}
The \emph{analytic rank} of $g$ is $r_\an(g)\defeq\ord_{k/2}L(g,s)\in\N$. 
\end{definition}

The number $\varepsilon(g)$ controls the parity of $r_\an(g)$; namely, $r_\an(g)$ is even if $\varepsilon(g)=1$ and is odd if $\varepsilon(g)=-1$. In other words, there is a congruence
\begin{equation} \label{analytic-congruence-eq}
r_\an(g)\equiv\frac{1-\varepsilon(g)}{2}\pmod{2}. 
\end{equation}
Equivalently, $\varepsilon(g)={(-1)}^{r_\an(g)}$.

\begin{remark} \label{analytic-rank-rem}
In order for the notion of analytic rank to make sense it is not necessary that the eigenform $g$ be a newform; in particular, $g$ may not satisfy a functional equation of the shape \eqref{L-functional-eq}. 
\end{remark}

\begin{remark}
The reader is referred, \emph{e.g.}, to \cite[Theorems 4.3.12 and 4.6.15]{Miyake} for the case of a newform on $\Gamma_0(M)$ whose character is not necessarily trivial.
\end{remark}

\subsubsection{Root numbers of Coleman families} \label{root-numbers-subsubsec}

Let $\f_U$ be a Coleman family as in \S \ref{coleman-subsubsec}. In the lines that follow, $k$ will always stand for an admissible weight in $U\cap\mathscr W_{N,\cla}^0$, with $\mathscr W_{N,\cla}^0$ as in \eqref{even-classical-weights-eq}. Let us write $s$ for the common value of the slopes $\ord_p\bigl(a_p(f_k)\bigr)$ (\emph{cf.} Remark \ref{slope-rem}); as pointed out in Remark \ref{slope-rem2}, we necessarily have $k\geq s+1$ and, conversely, all $k\in U\cap\mathscr W_{N,\cla}$ with $k>s+1$ are admissible. For every $k$, recall from \eqref{L-functional-eq} the functional equation
\[ \Lambda(f_k^\flat,s)=\varepsilon(f_k^\flat)\Lambda(f_k^\flat,k-s), \]
where $\varepsilon(f_k^\flat)\in\{\pm1\}$ is the root number of $f_k^\flat$. Let $\omega_{N,k}\in\{\pm1\}$ be the eigenvalue of the Atkin--Lehner involution acting on $f_k^\flat$; as remarked, \emph{e.g.}, in \cite[p. 170]{GSS}, the product 
\begin{equation} \label{omega-eq2}
\omega_N(\f_U)\defeq{(-1)}^{k/2}\omega_{N,k}\in\{\pm1\}
\end{equation}
is independent of $k$ and the equality 
\begin{equation} \label{epsilon-omega-eq}
\varepsilon(f_k^\flat)=\omega_N(\f_U)
\end{equation}
holds for every admissible $k\in U\cap\mathscr W_{N,\cla}^0$ (take the character $\chi$ in \cite[p. 170]{GSS} to be trivial). In particular, the root number $\varepsilon(f_k^\flat)$ is constant for all $k$ as above.

\begin{definition} \label{root-number-def}
The integer $\varepsilon(\f_U)\defeq\omega_N(\f_U)\in\{\pm1\}$ is the \emph{root number of $\f_U$}. 
\end{definition}

Again, it is worth bearing in mind that Definition \ref{root-number-def}, which might appear somewhat misleading in that $k$ on the right hand side of \eqref{omega-eq2} varies over $U\cap\mathscr W_{N,\cla}^0$, is justified by the fact that, as in \cite{GSS} and \cite{LZ}, here we consider specializations of even weight and trivial character only. 

As was done in \cite[\S 3.1]{Vigni} for Hida families, we introduce the following notion.

\begin{definition} \label{minimal-rank-def}
The \emph{minimal admissible generic rank} of $\f_U$ is
\[ r_{\min}(\f_U)\defeq\frac{1-\varepsilon(\f_U)}{2}. \] 
\end{definition}

Equivalently, $r_{\min}(\f_U)$ is the smallest analytic rank of an even weight specialization of $\f_U$ that is allowed by the functional equation: $r_{\min}(\f_U)=0$ if $\varepsilon(\f_U)=1$ and $r_{\min}(\f_U)=1$ if $\varepsilon(\f_U)=-1$.

\subsection{Greenberg's conjecture and consequences}

We want to recall the formulation, in the special case of Coleman families, of a ``visionary'' conjecture of R. Greenberg on analytic ranks in families of newforms.

\subsubsection{Greenberg's conjecture}

The conjecture we are about to state, which is concerned with the analytic rank of an even weight form in a Coleman family $\f_U$, is essentially due to Greenberg (\cite{Greenberg-CRM}). 

\begin{conjecture}[Greenberg] \label{greenberg-conj}
The equality
\[ r_\an(f_k^\flat)=r_{\min}(\f_U) \]
holds for all but finitely many admissible $k\in U\cap\mathscr W_{N,\cla}^0$.
\end{conjecture}

Actually, Greenberg's conjecture in its original formulation is much broader. Namely, it predicts that if $g$ varies over all normalized newforms of level $\Gamma_1(M)$ with $M$ divisible only by prime numbers in a fixed set (with no restriction on the weight), then $r_\an(g)$ should be as small as allowed by the functional equation, with at most finitely many exceptions. Thus, Conjecture \ref{greenberg-conj} is to be viewed as a special case of Greenberg's conjecture; notice that one of the main results of \cite{Vigni} treats an analogue of Conjecture \ref{greenberg-conj} for Hida families (\emph{i.e.}, for $p$-ordinary families). The reader is referred to \cite[p. 101, Conjecture]{Greenberg-CRM} for the original conjecture for newforms of weight $2$ and to the discussion following it for a conjecture for newforms of arbitrary weight. Evidence for Greenberg's conjecture is still very scarce: see, \emph{e.g.}, \cite[\S 3.2]{Vigni} for details. Here we would like to remark that results on the counterpart of Conjecture \ref{greenberg-conj} for Hida families have been obtained in rank $0$ and higher weight by Bertolini--Darmon (\cite[Corollary 4]{BD-rational}) and by Howard (\cite[Theorem 7]{Howard-derivatives}), while results in rank $1$ and weight $2$ have been proved by Howard (\cite{Howard-derivatives}). Finally, it is worth pointing out that results in the circle of ideas of Conjecture \ref{greenberg-conj} for Hida families can also be found in the recent paper \cite{disegni} by Disegni (see, in particular, \cite[Theorem G]{disegni}).

Under some technical assumptions (among which, the non-degeneracy of certain height pairings \emph{\`a la} Gillet--Soul\'e between Heegner cycles, which is assumed in \cite{Vigni} as well), in this paper we prove a result (Theorem B of the introduction) in the direction of Conjecture \ref{greenberg-conj}, both when $r_{\min}(\f_U)=0$ and when $r_{\min}(\f_U)=1$. As a matter of fact, we shall prove Conjecture \ref{greenberg-conj} for a $p$-adic family $\f_U$ where $p$ is allowed to vary over all but finitely many primes. %On the other hand, our result for a given $\f_U$ will be slightly stronger than predicted by Conjecture \ref{greenberg-conj} in that it will hold for \emph{all} admissible $k\in U\cap\mathscr W_{N,\cla}^0$.

\subsubsection{On algebraic $\fP$-ranks}

Recall the algebraic $\fP$-rank of a newform that was introduced in Definition \ref{algebraic-rank-def} (\emph{cf.} also Notation and terminology \ref{notation-rem}). Combining Conjecture \ref{greenberg-conj} with the conjectures of Birch--Swinnerton-Dyer (\cite[\S 1]{Tate-BSD}) and of Beilinson--Bloch--Kato (\cite[Conjecture 2.10]{LV}) on $L$-functions of modular forms, it is natural to propose 

\begin{conjecture} \label{main-conj}
The equality
\[ r_{\alg,\fP}\bigl(f_k^\flat/\Q\bigr)=r_{\min}\bigl(\f_U\bigr) \] 
holds for all but finitely many admissible $k\in U\cap\mathscr W_{N,\cla}^0$.
\end{conjecture}

See \cite[Conjecture 5.3]{Vigni} for an analogous conjecture for Hida families. Later in this paper we shall prove results (Corollaries \ref{sha-coro} and \ref{sha-coro2}) in the direction of Conjecture \ref{main-conj}.

\section{Weight $2$ forms, Heegner points and their $p$-adic variation}

In this section, we fix a newform $f$ of weight $2$ and square-free level $N$, then carefully introduce a set $\Xi_f$ that consists of all but finitely many prime numbers. For any $p\in\Xi_f$, we explain why there exists a $p$-adic Coleman family $\f_U$ passing through $f$ and how Heegner points can be interpolated along $\f_U$.

\subsection{The weight $2$ newform $f$}

Here we introduce the newform $f$ that will be one of our basic data. 

\subsubsection{Basic assumption}

Let $f\in S_2^\new(\Gamma_0(N))$ be a normalized newform of weight $2$ and level $N$ with $q$-expansion $f(q)=\sum_{n\geq1}a_n(f)q^n$. Throughout  this paper, we assume that 
\begin{itemize}
\item $N$ is square-free.
\end{itemize}
This condition, which could certainly be relaxed at the cost of adding extra technicalities to our main arguments, is introduced in order to simplify the exposition. In particular, it guarantees (see, \emph{e.g.}, \cite[p. 34]{ribet}) that $f$ has no complex multiplication in the sense of \cite[p. 34, Definition]{ribet}. Equivalently, the abelian variety $A_f$ attached to $f$ has no complex multiplication.

\begin{remark} \label{square-free-rem}
In most of the paper, the assumption that $N$ be square-free is not necessary, the only property of $f$ that is really needed being that $f$ have no complex multiplication. More precisely, we use the square-freeness of $N$ only in \S \ref{sha-subsec} (in particular, in Theorem \ref{sweeting-thm}). Therefore, the results we prove from here to \S \ref{sha-subsec} excluded apply to Coleman families of arbitrary (not necessarily square-free) tame level.
\end{remark}

\subsubsection{$p$-adic Galois image} \label{image-subsubsec}

Let $p$ be a prime number, let $\p$ vary over all the primes of $\Q_f$ above $p$ and recall from \S \ref{representations-subsubsec} the Galois representations 
\[ \rho_{f,p}:G_\Q\longrightarrow\GL_2(\cO_{\Q_f}\otimes_\Z\Z_p),\quad\rho_{f,\p}:G_\Q\longrightarrow\GL_2(\cO_{\Q_f,\p}) \]
attached to $f$. By \cite[Theorem 3.1]{ribet2}, for all but finitely many $p$ the image of $\rho_{f,p}$ contains the set 
\begin{equation} \label{determinant-eq}
\mathscr A_{f,p}\defeq\bigl\{A\in\GL_2(\cO_{\Q_f}\otimes_\Z\Z_p)\mid\det(A)\in\Z_p^\times\bigr\}. 
\end{equation}
Let $p$ be such a prime. Since there are splittings $\cO_{\Q_f}\otimes\Z_p=\oplus_{\p\mid p}\cO_{\Q_f,\p}$ and $\rho_{f,p}=\oplus_{\p\mid p}\rho_{f,\p}$, it follows that 
\begin{equation} \label{BI-eq}
\bigl\{A\in\GL_2(\Z_p)\mid\det(A)\in\Z_p^\times\bigr\}\subset\im(\rho_{f,\p}) \tag{BI}
\end{equation}
for every $\p$ above $p$ (this is a ``big image'' result for $\rho_{f,\p}$). In particular, \eqref{BI-eq} implies that $\SL_2(\Z_p)\subset\im(\rho_{f,\p})$ for every $\p$ above $p$. This fact will be used only later, but we find it convenient to record it here.

\subsubsection{Hecke polynomial at $p$} \label{hecke-polynomial-subsubsec}

From now on, we assume that $p\nmid2N$ (in \S \ref{heegner-subsec} we will refine our choice of $p$). Define $i(f),i_p(f)\in S_2(\Gamma_0(Np))$ by setting $i(f)(z)\defeq f(z)$ and $i_p(f)(z)\defeq f(pz)$ for all $z$ in the complex upper half-plane $\mathcal H$; in other words, $i(f)$ is the image of $f$ under the set-theoretic inclusion $S_2(\Gamma_0(N))\subset S_2(\Gamma_0(Np))$. The $\C$-vector subspace $\V_f$ of $S_2(\Gamma_0(Np))$ spanned by $i(f)$ and $i_p(f)$ is $2$-dimensional and the Hecke operator $U_p$ acts on $\V_f$ with characteristic polynomial
\begin{equation} \label{hecke-pol-eq} 
H_{f,p}(X)\defeq X^2-a_p(f)X+p
\end{equation}
(here we are using the fact that $f$ has weight $2$ and trivial character). We call $H_{f,p}$ the \emph{Hecke polynomial of $f$ at $p$}.

\begin{theorem}[Coleman--Edixhoven] \label{hecke-thm}
The polynomial $H_{f,p}$ has distinct roots.
\end{theorem}

\begin{proof} This is \cite[Theorem 2.1]{CE}. \end{proof}

\begin{remark}
It is well known that if $\lambda$ is a prime of $\Q_f$ not dividing $p$ and $\rho_{f,\lambda}$ is the $\lambda$-adic representation of $G_\Q$ attached to $f$, then $H_{f,p}$ is also the characteristic polynomial of $\rho_{f,\lambda}(\Frob_p)$, where $\Frob_p\in G_\Q$ is a Frobenius element at $p$. 
\end{remark}

\subsubsection{$p$-stabilizations and Coleman families} \label{stabilizations-subsubsec}

Recall the Hecke polynomial $H_{f,p}$ from \eqref{hecke-pol-eq} and let $\alpha,\beta\in\bar\Z$ be its roots (by Theorem \ref{hecke-thm}, we know that $\alpha\not=\beta$, but this fact will play no role in the next few lines). Let us define the \emph{$p$-stabilizations} $f_\alpha$ and $f_\beta$ of $f$ by
\[ f_\alpha(z)\defeq f(z)-\beta f(pz),\quad f_\beta(z)\defeq f(z)-\alpha f(pz) \]
for all $z\in\mathcal H$. It can be checked that $f_\alpha$ and $f_\beta$ are normalized cusp forms of level $\Gamma_0(Np)$ that are eigenforms for the Hecke operators $T_n$ with $(n,Np)=1$; in this case, the eigenvalues of $T_n$ with eigenvectors $f_\alpha$ or $f_\beta$ coincide with those of $f$. Furthermore, $f_\alpha$ and $f_\beta$ are also eigenvalues for the Hecke operator $U_p$; more precisely, they satisfy 
\[ U_p(f_\alpha)=\alpha f_\alpha,\quad U_p(f_\beta)=\beta f_\beta. \]
Since $\ord_p(\alpha)+\ord_p(\beta)=1$, we infer that either $f_\alpha$ or $f_\beta$ has critical slope (\emph{cf.} Remark \ref{critical-rem}) if and only if $0\in\{\ord_p(\alpha),\ord_p(\beta)\}$, and both conditions are equivalent to $\ord_p\bigl(a_p(f)\bigr)=0$, \emph{i.e.}, to $f$ being ordinary at $p$, because $a_p(f)=\alpha+\beta$. The case of Hida families was studied in \cite{Vigni}, so in this paper we work under

\begin{assumption} \label{0-ass}
$0\notin\bigl\{\ord_p(\alpha),\ord_p(\beta)\bigr\}$.
\end{assumption}

Relabelling if necessary, we also assume that $\ord_p(\alpha)\leq\ord_p(\beta)$. From now on we set $f^\sharp\defeq f_\alpha$ and simply refer to $f^\sharp$ as the $p$-stabilization of $f$.

\begin{proposition} \label{coleman-existence-prop}
There exists a Coleman family passing through $f^\sharp$.
\end{proposition}

\begin{proof} Of course, $(f^\sharp)^\flat=f$, so Theorem \ref{hecke-thm} ensures that condition (1) in Definition \ref{noble-def} is satisfied by $f^\sharp$. On the other hand, $a_p(f^\sharp)=\alpha$ and, by Assumption \ref{0-ass}, $\ord_p(\alpha)\neq1$; hence, property (2) in Definition \ref{noble-def} is trivially enjoyed by $f^\sharp$. We conclude that $f^\sharp$ is a noble eigenform, so the existence of a Coleman family through $f^\sharp$ follows from Theorem \ref{interpolation-thm}. \end{proof}

As before, we write $\f_U$ for the Coleman family through $f$ (or, rather, $f^\sharp$) from Proposition \ref{coleman-existence-prop}. When we want to emphasize the prime $p$, as we will need to do later in this paper, we write $\hf_U$ in place of $\f_U$. Note that, here, $U$ is a suitable open disc in $\mathscr W_N$ such that $2\in U$.

\subsubsection{Assumption on $r_\an(f)$} \label{analytic-assumption-subsubsec}

From here until Section \ref{rank-zero-sec}, the following assumption will be in force.

\begin{assumption} \label{main-ass}
$r_\an(f)=1$.
\end{assumption}

In particular, by \eqref{analytic-congruence-eq}, the root number of $f$ is $\varepsilon(f)=-1$.

\begin{remark}
By \cite[Lemma 2.9]{Vigni}, there is an equality $r_\an(f)=r_\an(f^\sharp)$, so $r_\an(f^\sharp)=1$.
\end{remark}

\subsection{Heegner points and choice of $p$} \label{heegner-subsec}

As in \cite[\S 4.3]{Vigni}, the choice of $p$ is made in terms of Heegner points on $A_f$. 

\subsubsection{Heegner points} \label{heegner-subsubsec}

Recall that, as a consequence of Assumption \ref{main-ass}, $\varepsilon(f)=-1$. By a result of Waldspurger (\cite{Waldspurger}) and Bump--Friedberg--Hoffstein (\cite[p. 543, Theorem, (ii)]{BFH}), there exists an imaginary quadratic field $K$, whose associated Dirichlet character will be denoted by $\chi_K$, such that
\begin{itemize}
\item[(a)] the primes dividing $N$ split in $K$;
\item[(b)] $r_\an(f\otimes\chi_K)=0$.
\end{itemize}
Fix such a field $K$ once and for all. Condition (a) can be equivalently formulated by saying that $K$ satisfies the \emph{Heegner hypothesis} relative to $N$; as in \cite[\S 4.3]{Vigni}, let $\alpha_1\in A_f(K_1)$ be a Heegner point on $A_f$ rational over the Hilbert class field of $K$ and let
\begin{equation} \label{alpha-K-eq}
\alpha_K\defeq\tr_{K_1/K}(\alpha_1)\in A_f(K)
\end{equation}
be the $K$-rational Heegner point that appears in the work of Kolyvagin on the arithmetic of elliptic curves and in its extensions to modular abelian varieties (see, \emph{e.g.}, \cite[\S 2.3]{Kol-Log}). Since there is a factorization 
\[ L(f/K,s)=L(f,s)\cdot L(f\otimes\chi_K,s), \]
Assumption \ref{main-ass} and condition (b) above give $r_\an(f/K)=1$, which amounts, by the Gross--Zagier formula (\cite[Ch. I, Theorem 6.3]{GZ}), to $\alpha_K$ being non-torsion.

\subsubsection{Choice of $p$} \label{choice-subsubsec}

For every prime ideal $\lambda$ of $\cO_f$ denote by $\bar\rho_{f,\lambda}$ the representation of $G_\Q$ over the residue field of $\Q_f$ at $\lambda$ associated with $f$. We will exploit the following two facts:
\begin{itemize}
\item for all but finitely many prime ideals $\p$ of $\cO_f$, the point $\alpha_K$ in \eqref{alpha-K-eq} does not belong to the $\cO_f$-submodule $\p A_f(K)$ (\cite[Proposition 4.8]{Vigni});
\item for all but finitely many prime ideals $\lambda$ of $\cO_f$, the representation $\bar\rho_{f,\lambda}$ is irreducible (\cite[Theorem 2.1, (a)]{ribet2}).
\end{itemize}
Write $\mathscr P$ for the set of all prime numbers. Let $\p_1,\dots,\p_n$ be the prime ideals of $\cO_f$ such that $\alpha_K\in\p_iA_f(K)$ for $i\in\{1,\dots,n\}$, denote by $p_i$ the residue characteristic of $\p_i$ and define
\[ \mathscr P_f\defeq\mathscr P\smallsetminus\{p_1,\dots,p_n\}. \]
On the other hand, let $\{\ell_1,\dots,\ell_t\}$ be the set of the residue characteristics of those primes $\lambda$ for which $\bar\rho_{f,\lambda}$ is reducible and set
\[ \mathscr I_f\defeq\mathscr P\smallsetminus\{\ell_1,\dots,\ell_t\}. \]

\begin{remark} \label{absolutely-rem}
If the residue characteristic of the prime $\lambda$ is not $2$, then $\bar\rho_{f,\lambda}$ is irreducible if and only if it is absolutely irreducible (see, \emph{e.g.}, \cite[(1.5.3), (3)]{NP}).
\end{remark}

Recall from \S \ref{image-subsubsec} that there is a finite subset $\mathscr S_f\subset\mathscr P$ such that for every $p\in\mathscr S_f$ the image of $\rho_{f,p}$ does not contain the set $\mathscr A_{f,p}$ from \eqref{determinant-eq}. Finally, define
\[ \mathscr B_f\defeq\mathscr P\smallsetminus\mathscr S_f \]
and
\begin{equation} \label{xi-f-eq}
\Xi_f\defeq\bigl\{\ell\in\mathscr P\mid\ell\nmid6N\bigr\}\cap\mathscr P_f\cap\mathscr I_f\cap\mathscr B_f. 
\end{equation}
By what we have just said, $\Xi_f$ excludes only finitely many prime numbers. 

\begin{remark}
Fix an odd prime $\ell$ and an integer $M\geq1$ with $\ell\nmid M$. It is a theorem of Jochnowitz (\cite{jochnowitz-TAMS}) that there are only finitely many mod $\ell$ modular
representations of level $\Gamma_1(M)$ (see, \emph{e.g.}, \cite[Proposition 5.1.1]{ColMaz} for the case of level $\Gamma_1(M\ell^r)$ for some $r\in\N$). On the other hand, by a recent result of Calegari--Sardari (\cite[Theorem 1.1]{CS}), there are only finitely many eigenforms $g$ of even weight and level $M$ without complex multiplication such that $a_\ell(g)=0$. 
\end{remark}

Now pick a prime $p\in\Xi_f$, which should be regarded as fixed until further notice, and fix a $p$-adic Coleman family $\hf_U$ passing through $f$, which exists by Proposition \ref{coleman-existence-prop}.

\subsection{Big Heegner classes and their specializations}

As before, for every admissible weight $k\in U\cap\mathscr W_{N,\cla}^0$ we denote by $f_k^\flat$ the newform of level $N$ of which the specialization $f_k$ of $\f_U$ is one of the two $p$-stabilizations. 

\subsubsection{Heegner classes} \label{heegner-classes-subsubsec}

Let $k\in U\cap\mathscr W_{N,\cla}^0$ such that $k>2$ and let $K$ be the imaginary quadratic field that was chosen in \S \ref{heegner-subsubsec}. In \cite{Nek}, Nekov\'a\v{r} introduced a collection of (higher weight) Heegner classes
\begin{equation} \label{heegner-classes-eq}
y_{k,c}\in H^1\Bigl(K_c,T_{f_k^\flat}^\dag\Bigr) 
\end{equation}
that are indexed by the integers $c\geq1$ such that $(c,pND_K)=1$; here $T_{f_k^\flat}^\dag=T_{f_k^\flat,\p}^\dag$ is the $G_\Q$-stable $\cO_{\Q_{f_k^\flat},\p}$-lattice defined in \S \ref{Galois_repr}. By definition, these classes are images under \'etale Abel--Jacobi maps of Heegner cycles on Kuga--Sato varieties that are built out of Heegner points on the modular curve $X_0(N)$. In particular, it follows from a result of Nizio\l, Nekov\'a\v{r} and Saito (see, \emph{e.g.}, \cite[Theorem 2.4]{LV}) that
\[ y_{k,c}\in H^1_f\Bigl(K_c,T_{f_k^\flat}^\dag\Bigr), \]
where $H^1_f(K_c,\star)$ stands for the relevant Bloch--Kato Selmer group (see, \emph{e.g.}, \cite[\S 2.4]{LV}).

\begin{remark} \label{AJ-Kummer-rem}
For $k=2$, the class $y_{2,1}\in H^1_f\bigl(K_1,T_f^\dag\bigr)\simeq H^1_f\bigl(K_1,\Ta_\p(A_f)\bigr)$ turns out to be the Kummer image of the $f$-isotypical projection of a Heegner point on the Jacobian $J_0(N)$ of $X_0(N)$. In fact, as observed in \cite[Examples 2.2]{Nek3},  $\CH^1(X_0(N)/K_1{)}_0$ coincides with $\Pic^0(X_0(N))(K_1)$, which is the Mordell--Weil group $J_0(N)(K_1)$, and the cycle class map reduces to the Kummer map
\[ J_0(N)(K_1)\otimes_\Z\cO_{\Q_f,\p}\longrightarrow H^1\Bigl(K_1,H^1_{\text{\'et}}\bigl(X_0(N)_{\bar{\Q}},\cO_{\Q_f,\p}(1)\bigr)\!\Bigr). \]
Then it can be checked that taking the $f$-isotypical quotient produces the weight $2$ version of the \'etale Abel--Jacobi map $\Phi_{f,K_1,\p}$ from \S \ref{AJ-subsubsec}.
\end{remark}

\subsubsection{The class $\zeta_{\hf,c}$} \label{big-heegner-subsubsec}

Recently, Jetchev--Loeffler--Zerbes (\cite{JLZ}) and B\"uy\"ukboduk--Lei (\cite{BL}) have proved independently that the classes in \eqref{heegner-classes-eq} can be interpolated along $\f_U$; in this paper, we find it more convenient to use the results of \cite{BL}. For every integer $c\geq1$ coprime to $Np$, B\"uy\"ukboduk and Lei construct a big Heegner-type class
\begin{equation} \label{zeta-integrality-eq}
\zeta_{\f_U,c}\in H^1\Bigl(K_c,\T^\dagger_{\f_U}\Bigr), 
\end{equation}
where, as in \S \ref{big-galois-subsec}, $\T^\dagger_{\f_U}$ is the (twisted) integral big Galois representation attached to $\f_U$. The definition of $\zeta_{\f_U,c}$, which comes from a universal class interpolating generalized Heegner cycles \emph{\`a la} Bertolini--Darmon--Prasanna (\cite{BDP}), is essentially immaterial for our goals: see \cite[\S 4.1]{BL} for details (\emph{cf.} \S \ref{integrality-subsubsec} too).

\subsubsection{On the integrality of $\zeta_{\f_U,c}$} \label{integrality-subsubsec}

The integrality of the big Heegner class $\zeta_{\f_U,c}$ is key for our purposes and we want to highlight that it holds unconditionally; the ingredients needed to check this property are contained in \cite[Appendix B]{BL-bis}. 

With notation and terminology analogous to those of \cite{BL}, which we will use freely in the following lines, $\zeta_{\f_U,c}$ is the evaluation at the trivial character of the universal Heegner class $\zeta_{\f_U,c}^{\mathrm{ac}}$ interpolating the generalized Heegner classes of Bertolini--Darmon--Prasanna along the Coleman family $\f_U$. Set $K_{p^\infty}\defeq\cup_{n\geq1}K_{p^n}$ and $\tilde{\Gamma}_{\mathrm{ac}}\defeq\Gal(K_{p^\infty}/K)\simeq\Z_p\times\Delta$ with $\Delta$ finite, fix a topological generator $\gamma_{\mathrm{ac}}$ of the free part of $\tilde{\Gamma}_{\mathrm{ac}}$ and write $\Lambda(\tilde{\Gamma}_{\mathrm{ac}})$ for the Iwasawa algebra of $\tilde{\Gamma}_{\mathrm{ac}}$ over $\Z_p$. Finally, set $\Lambda(\tilde{\Gamma}_{\mathrm{ac}})_{\Q_p}\defeq\Lambda(\tilde{\Gamma}_{\mathrm{ac}})\otimes_{\Z_p} \Q_p$. It turns out that  
\[ \zeta_{\f_U,c}^{\mathrm{ac}}\in H^1\Bigl(K_c,\T_{\f_U}^{\dag,\mathrm{ac}}[1/p]\Bigr)\hat{\otimes}_{{\Lambda(\tilde{\Gamma}_{\mathrm{ac}})}_{\Q_p}}\mathcal{H}(\tilde{\Gamma}_{\mathrm{ac}}), \]
where $\T_{\f_U}^{\dag,\mathrm{ac}}$ is the anticyclotomic twist of $\T_{\f_U}^\dag$ with respect to the canonical involution on $\tilde{\Gamma}_{\mathrm{ac}}$ (\emph{cf.} \cite[\S 1.1.2]{BL}) and $\mathcal{H}(\tilde{\Gamma}_{\mathrm{ac}})$ is an appropriate ${\Lambda(\tilde{\Gamma}_{\mathrm{ac}})}_{\Q_p}$-submodule of the $\Q_p$-algebra of power series $\Q_p[[\gamma_{\mathrm{ac}}-1]]$. 

If the image of a certain big Dieudonn\'e module under Perrin-Riou's big exponential map is an integral class (this is condition {\bf (IntPR)} in \cite[Appendix]{BL}, which always holds true in the slope $0$ setting, \emph{i.e.}, for Hida families), then
\begin{equation} \label{int-eq}
\zeta_{\f_U,c}^{\mathrm{ac}}\in H^1\Bigl(K_c,\T_{\f_U}^{\dagger,\mathrm{ac}}\Bigr)\hat\otimes_{\Lambda(\tilde{\Gamma}_{\mathrm{ac}})}\,p^C\mathcal{H}^+(\tilde{\Gamma}_{\mathrm{ac}}),
\end{equation}
where $C$ is a constant that depends only on the slope of $\f_U$ and $\mathcal{H}^+(\tilde{\Gamma}_{\mathrm{ac}})\subset \mathcal{H}(\tilde{\Gamma}_{\mathrm{ac}})$ is the $\Lambda(\tilde{\Gamma}_{\mathrm{ac}})$-submodule of power series that are integral with respect to a suitable valuation (\emph{cf.} \cite[\S 1.1.1]{BL}). Now let $\Hom_{\mathrm{cts}}(\tilde{\Gamma}_{\mathrm{ac}},\GG_m)$ be the space in which the anticyclotomic characters of the generalized Heegner classes can vary. In particular, one deduces from \eqref{int-eq} that 
\begin{equation} \label{zeta-integrality-eq2}
\mathds{1}\bigl(\zeta_{\f_U,c}^{\mathrm{ac}}\bigr)\in H^1\Bigl(K_c,\T^\dagger_{\f_U}\Bigr), 
\end{equation}
where $\mathds{1}\in\Hom_{\mathrm{cts}}\bigl(\tilde\Gamma_{\mathrm{ac}},\GG_m\bigr)$ denotes the trivial character. As explained in \cite[Appendix B]{BL-bis}, the integrality \eqref{int-eq} of $\zeta_{\f_U,c}^{\mathrm{ac}}$ is unconditional as long as one replaces $\Hom_{\mathrm{cts}}(\tilde{\Gamma}_{\mathrm{ac}},\GG_m)$ with any wide open disc around $\mathds{1}$ of radius smaller than $1$ (\emph{cf.}, in particular, \cite[Theorem B.5]{BL-bis}). On the other hand, by definition, $\zeta_{\f_U,c}\defeq\mathds{1}\bigl(\zeta_{\f_U,c}^{\mathrm{ac}}\bigr)$, so the integrality of $\zeta_{\f_U,c}$ as in \eqref{zeta-integrality-eq} follows from \eqref{zeta-integrality-eq2}.

\subsubsection{The specialization theorem} \label{specialization-subsubsec}

As explained in \S \ref{big-galois-subsec}, for every admissible $k\in U\cap\mathscr{W}_{N,\text{cl}}$ there is a Galois-equivariant specialization map 
\[ \spe_k:\T_{\f_U}^\dagger\longrightarrow T_{f_k}^{*,\dag}. \] 
On the other hand, as observed in \cite[\S 3.1]{BL}, there is a natural isomorphism 
\begin{equation} \label{pr-eq}
V_{f_k}^{*,\dag}\overset\simeq\longrightarrow V_{f_k^\flat}^{*,\dag}.
\end{equation}
By the argument employed at the end of \S \ref{residual-subsubsec-1}, up to rescaling by a suitable element of $\Q_{f_k^\flat,\fP}^\times$ we can assume that \eqref{pr-eq} induces a $G_\Q$-equivariant isomorphism
\begin{equation} \label{pr-eq2}
T_{f_k}^{*,\dag}\overset\simeq\longrightarrow T_{f_k^\flat}^{*,\dag}
\end{equation}
between lattices. Finally, by \eqref{Poinc} there is an isomorphism $T_{f_k^\flat}^{*,\dag}\simeq T_{f_k^\flat}^\dag$. By functoriality, for every number field $L$ there is a specialization map in cohomology  
\begin{equation} \label{specialization-eq}
\spe_{k,L}: H^1\Bigl(L,\T_{\f_U}^\dag\Bigr)\longrightarrow H^1\Bigl(L,T_{f_k^\flat}^\dag\Bigr). 
\end{equation}
We remark that the notation used for the map in \eqref{specialization-eq} is different from that of \cite{BL} and closer to the one adopted in \cite{Vigni}.

Let $y_{k,1}$ be the Heegner class from \eqref{heegner-classes-eq}; let us introduce the $K$-rational Heegner class
\begin{equation} \label{heegner-K-eq}
y_{k,K}\defeq\cores_{K_1/K}(y_{k,1})\in H^1\Bigl(K,T_{f_k^\flat}^\dag\Bigr). 
\end{equation}
Finally, consider also the $K$-rational big Heegner class
\[ \zeta_{\f_U,K}\defeq\cores_{K_1/K}\bigl(\zeta_{\f_U,1}\bigr)\in H^1\Bigl(K,\T_{\f_U}^\dag\Bigr). \]
Now we can state the specialization theorem we are interested in, which is a consequence of the main result of \cite{BL}.

\begin{theorem}[B\"uy\"ukboduk--Lei] \label{BL-thm}
For every admissible $k\in U\cap\mathscr W_{N,\cla}^0$, there is $\ell_k\in\bar\Z_p^\times$ such that
\[ \spe_{k,K}\bigl(\zeta_{\f_U,K}\bigr)=\ell_k\cdot y_{k,K}. \]
\end{theorem}

\begin{proof} This is \cite[Proposition 4.15, (ii)]{BL}. To see that the constant $\ell_k$ is indeed a $p$-adic unit, recall that $f_k^\flat\in S_k^\new(\Gamma_0(N))$, \emph{i.e.}, $f_k^\flat$ is a newform of level $N$ and trivial character. Thus, the pseudo-eigenvalue $\lambda_N\bigl(f_k^\flat\bigr)$ of the Atkin--Lehner operator $W_N$ is actually an eigenvalue (see, \emph{e.g.}, \cite[p. 224]{AL}), and then $\lambda_N\bigl(f_k^\flat\bigr)\in\{\pm1\}$ because $W_N$ is an involution on $S_k^\new(\Gamma_0(N))$. Now the fact that $\ell_k$ is a $p$-adic unit follows from the explicit expression for $\ell_k$ given in \cite[Proposition 4.15, (ii)]{BL}. \end{proof}

\begin{remark}
The reader is referred to work of Castella (\cite{CasHeeg}) and of Ota (\cite{ota-JNT}) for results analogous to Theorem \ref{BL-thm} in the setting of Hida families.
\end{remark}

\section{Shafarevich--Tate groups in Coleman families: the rank one case} \label{sha-sec}

In this section, we prove our results on Shafarevich--Tate groups and on images of $p$-adic \'etale Abel--Jacobi maps attached to a large class of even weight eigenforms in the Coleman family $\hf_U$ introduced at the end of \S \ref{choice-subsubsec}, and so in a rank $1$ setting.

\subsection{Torsion-freeness in cohomology} \label{image-subsec2}

Here we want to prove a torsion-freeness result for the cohomology of the self-dual Galois representations attached to the specializations of the Coleman family $\f_U=\hf_U$ that was fixed at the end of \S \ref{choice-subsubsec}. The result we obtain is the counterpart in the finite slope setting of \cite[Corollary 4.24]{Vigni}. It is worth remarking that the strategy in \cite{Vigni} towards torsion-freeness in cohomology was crucially based on results of Fischman on the image of big Galois representations associated with Hida families (\cite{fischman}). Unfortunately, no analogue for Coleman families of Fischman's results seems to be currently available (but see, \emph{e.g.}, \cite[\S 10.5]{CLM} for results in this direction), so in this paper we adopt an alternative, more direct approach that takes advantage of our ``big image'' assumption on $f$ (\emph{cf.} \S \ref{image-subsubsec} and \eqref{xi-f-eq}).

\subsubsection{Representations of non-solvable type} \label{non-solvable-subsubsec}

As in \cite[\S 4.5]{Vigni}, a group representation over a commutative ring is said to be \emph{of non-solvable type} if its image is a non-solvable group. As observed in \cite[Remark 1.1, (i)]{BL}, the fact that, thanks to our choice of $p$, the image of $\rho_{f,p}$ contains the set $\mathscr A_{f,p}$ from \eqref{determinant-eq} ensures that $\bar\rho_f^\dagger$ is full, \emph{i.e.}, the image of $\bar\rho^\dagger_f$ contains a conjugate of $\SL_2(\F_p)$; in particular, $\bar\rho^\dagger_f$ is of non-solvable type (\emph{cf.} \cite[Lemma 4.15]{Vigni}). Furthermore, as was pointed out in \S \ref{residual-subsubsec}, $\bar\rho^\dagger_f$ is equivalent to $\bar\rho^\dagger_{f_k}$ for every admissible weight $k\in U\cap\mathscr W_{N,\cla}$ with $k\equiv2\pmod{2(p-1)}$, so $\bar\rho^\dagger_{f_k}$ is of non-solvable type for all such $k$.

\subsubsection{Torsion-freeness of $H^1$ over solvable extensions}

We begin with two auxiliary results.

\begin{lemma} \label{lemma-1}
Let $\rho$ be an irreducible representation of $G_\Q$ of non-solvable type. If $F\subset\bar\Q$ is a solvable field extension of $\Q$, then $H^0(F,\rho)$ is trivial.
\end{lemma}

\begin{proof} This is essentially \cite[Lemma 3.10]{LV}. \end{proof}

\begin{lemma} \label{lemma-2}
Let $G$ be a profinite group, let $\cO$ be the valuation ring of a finite extension of $\Q_p$, let $\varpi$ be a uniformizer of $\cO$ and let $T$ be a free $\cO$-module of finite rank equipped with a continuous $\cO$-linear action of $G$. If $H^0(G,T/\varpi T)$ is trivial, then the $\cO$-module $H^1(G,T)$ is torsion-free.
\end{lemma}

\begin{proof}  This follows easily by considering the long exact sequence in cohomology associated with the short exact sequence of Galois modules
\[ 0\longrightarrow T\overset{\varpi\cdot}\longrightarrow T\longrightarrow T/\varpi T\longrightarrow0, \]
where the first non-trivial arrow is the multiplication-by-$\varpi$ map. \end{proof}

Now we are ready to prove the cohomological result that will be used in our study of the (algebraic and analytic) ranks of the specializations of $\f_U$. In the next statement, $F$ is a number field.

\begin{proposition} \label{torsion-free-prop}
If $F/\Q$ is solvable and $k\in U\cap\mathscr W_{N,\cla}$ is an admissible weight such that $k\equiv2\pmod{2(p-1)}$, then the $\cO_{\Q_{f_k},\fP}$-module $H^1\bigl(F,T^\dagger_{f_k}\bigr)$ is torsion-free.
\end{proposition}

\begin{proof} By our choice of $p$, the representation $\bar\rho_f$ is irreducible; since the irreducibility of a given representation is preserved by tensorization with $1$-dimensional
representations (see, \emph{e.g.}, \cite[Exercise 2.2.14, (2)]{kowalski}), the same is true of $\bar\rho_f^\dagger$. On the other hand, as we remarked before, $\bar\rho_f^\dagger$ is of non-solvable type, so $H^0\bigl(F,\bar T_f^\dagger\bigr)$ is trivial by Lemma \ref{lemma-1}. Finally, as explained in \S \ref{residual-subsubsec}, if $k\in U\cap\mathscr W_{N,\cla}$ is admissible and $k\equiv2\pmod{2(p-1)}$, then $\bar T_f^\dagger\simeq\bar T_{f_k}^\dagger$ as $G_\Q$-representations, from which we deduce that $H^0\bigl(F,\bar T_{f_k}^\dagger\bigr)$ is trivial. Now the torsion-freeness of the $\cO_{f_k,\fP}$-module $H^1\bigl(F,T^\dagger_{f_k}\bigr)$ follows from Lemma \ref{lemma-2}. \end{proof}

\subsection{Non-triviality of Heegner cycles in families}

Pick $p\in\Xi_f$, where $\Xi_f$ is the set of prime numbers introduced in \eqref{xi-f-eq}, and let $\p$ a prime of $\Q_f$ above $p$. As before, let $\fP$ be our fixed prime ideal of $\bar\Z$ above $p$.

\subsubsection{Kummer maps}

Write $\Ta_\p(A_f)$ for the $\p$-adic Tate module of $A_f$ and $A_f[\p]$ for the $\p$-torsion $\cO_f$-submodule of $A_f(\bar\Q)$. For any number field $F$, the surjection $\Ta_\p(A_f)\twoheadrightarrow A_f[\p]$ induces a map in Galois cohomology $\pi_{f,F}:H^1\bigl(F,\Ta_\p(A_f)\bigr)\rightarrow H^1(F,A_f[\p])$. There are also Kummer maps $\delta_{f,F}:A_f(F)\rightarrow H^1\bigl(F,\Ta_\p(A_f)\bigr)$ and $\pi_{f,F}\circ\delta_{f,F}: A_f(F)\rightarrow H^1(F,A_f[\p])$ (see, \emph{e.g.}, \cite[Appendix A.1]{GP}), the second of which gives an injection
%\begin{equation} \label{delta-bar-eq}
\[ \bar\delta_{f,F}:A_f(F)\big/\p A_f(F)\longmono H^1(F,A_f[\p]). \] 
%\end{equation}
Let $K$ be the imaginary quadratic field from \S \ref{heegner-subsubsec}. Thanks to the choice of $p$ we made in \S \ref{choice-subsubsec}, the image $\bar\alpha_K$ of $\alpha_K$ in $A_f(K)/\p A_f(K)$ is non-zero, hence 
\begin{equation} \label{nonzero-eq}
\bar\delta_{f,K}(\bar\alpha_K)\not=0. 
\end{equation}

\subsubsection{Non-triviality of Heegner cycles} \label{nontrivial-subsubsec}

With notation as in \eqref{classical-weights-n-eq}, let us fix an admissible $k\in U\cap\mathscr W_{N,\cla,4}$ with $k\equiv 2 \pmod{2(p-1)}$ and recall the Heegner cycle $y_{k,K}$ from \eqref{heegner-K-eq}. The next result is the counterpart in our finite slope setting of \cite[Theorem 5.17]{Vigni}.

\begin{theorem} \label{heegner-nontrivial-prop}
The cycle $y_{k,K}$ is not torsion over $\cO_{\Q_{f_k^\flat},\fP}$.
\end{theorem}

\begin{proof} In light of Remark \ref{AJ-Kummer-rem}, one can identify the cohomology classes $y_{2,K}\in H^1\bigl(K,T_f^\dag\bigr)$ and $\delta_{f,K}(\alpha_K)\in H^1\bigl(K,\Ta_\p(A_f)\bigr)$. Furthermore, reduction modulo $\fP$ (equivalently, in weight $2$, modulo $\p$) induces maps 
\[ \varpi_{2,K}:H^1\bigl(K,T_f^\dagger\bigr)\longrightarrow H^1\bigl(K,\bar T_f^\dagger\bigr),\quad\varpi_{k,K}:H^1\Bigl(K,T_{f_k^\flat}^\dagger\Bigr)\longrightarrow H^1\Bigl(K,\bar T_{f_k^\flat}^\dagger\Bigr), \]
the first of which allows one to identify $\varpi_{2,K}(y_{2,K})$ and $\bar\delta_{f,K}(\bar\alpha_K)$. It follows from \eqref{nonzero-eq} that 
\begin{equation} \label{y-nonzero-eq}
\varpi_{2,K}(y_{2,K})\neq0.
\end{equation} 
On the other hand, diagram \eqref{diagram} induces a commutative diagram 
\begin{equation} \label{cohomology-diagram-eq}
\xymatrix@C=41pt@R=25pt{
H^1\bigl(K,T_f^\dag\bigr)\ar[d]^-{\varpi_{2,K}} & H^1\Bigl(K,\T_{\f_U}^\dagger\Bigr)\ar[l]_-{\spe_{2,K}}\ar[r]^-{\spe_{k,K}} & H^1\Bigl(K,T_{f_k^\flat}^\dag\Bigr)\ar[d]^-{\varpi_{k,K}}\\
H^1\bigl(K,\bar{T}_f^\dag\bigr)\ar[rr]_-{\simeq}^-{\overline{\spe}} && H^1\Bigl(K,\bar{T}_{f_k^\flat}^\dag\Bigr)
}
\end{equation}
in which $\spe_{2,K}$ and $\spe_{k,K}$ are the specialization maps from \eqref{specialization-eq} and $\overline\spe$ is an isomorphism. The commutativity of \eqref{cohomology-diagram-eq} and Theorem \ref{BL-thm} yield equalities
\begin{equation} \label{equalities-eq}
\begin{split}
\overline\spe\bigl(\varpi_{2,K}(y_{2,K})\bigr)&= \overline\spe\Bigl(\varpi_{2,K}\bigl(\ell_2^{-1}\cdot\spe_{2,K}(\zeta_{\f_U,K})\bigr)\!\Bigr)\\
 &=\bar{\ell}_2^{-1}\cdot\varpi_{k,K}\bigl(\spe_{k,K}(\zeta_{\f_U,K})\bigr)=\bar{\ell}_2^{-1}\bar{\ell}_k\cdot\varpi_{k,K}(y_{k,K})
\end{split}
\end{equation}
with $\bar{\ell}_2^{-1}\bar{\ell}_k\in\bar\F_p^\times$, and then \eqref{y-nonzero-eq} and \eqref{equalities-eq} show that $y_{k,K}\not=0$. Finally, keeping isomorphism \eqref{pr-eq2} in mind, the $\cO_{\Q_{f_k^\flat},\fP}$-module $H^1\bigl(K,T_{f_k^\flat}^\dagger\bigr)$ is torsion-free by Proposition \ref{torsion-free-prop}, and the proof is complete. \end{proof}

\subsection{On Shafarevich--Tate groups in families} \label{sha-subsec}

Under our running Assumption \ref{main-ass}, we prove a result on algebraic ranks and Shafarevich--Tate groups of specializations of $\hf_U$ for the prime $p\in\Xi_f$ that was chosen in \S \ref{choice-subsubsec}.

\subsubsection{Sweeting's converse theorem}

Recall that the dimension $d_f$ of the abelian variety $A_f$ attached to $f$ is equal to $[\cO_f:\Z]$. Let $\Sha_{p^\infty}(A_f/\Q)$ (respectively, $\Sha_{\p}(A_f/\Q)$) be the $p$-primary (respectively, $\p$-primary) Shafarevich-Tate group of $A_f$ over $\Q$. 

The result we are about to state generalizes converse theorems of, among others, Skinner (\cite[Theorem A]{Skinner}) and W. Zhang (\cite[Theorem 1.5]{zhang-selmer}).

\begin{theorem}[Sweeting] \label{sweeting-thm}
If $\mathrm{rank}_\Z\, A_f(\Q)=d_f$ and $\Sha_{\p}(A_f/\Q)$ is finite, then $r_\an(f)=1$. Conversely, if $r_\an(f)=1$, then $\mathrm{rank}_\Z\, A_f(\Q)=d_f$ and $\Sha_{p^\infty}(A_f/\Q)$ is finite.
\end{theorem}

\begin{proof} Since 
\begin{itemize}
\item $N$ is square-free, 
\item $\bar\rho_{f,\p}$ is (absolutely) irreducible (\emph{cf.} \S \ref{choice-subsubsec}),
\item the image of $\rho_{f,\p}$ contains $\SL_2(\Z_p)$ (\emph{cf.} \eqref{BI-eq} in \S \ref{image-subsubsec}), 
\end{itemize}
this is \cite[Corollary D]{sweeting}. \end{proof}

By Theorem \ref{sweeting-thm}, we can reformulate Assumption \ref{main-ass} as

\begin{assumption} \label{main-ass2}
$\rank_\Z A_f(\Q)=d_f$ and $\#\Sha_{\p}(A_f/\Q)<\infty$.
\end{assumption}

This is the condition that will be used in \S \ref{alg-sha-subsubsec}.

\begin{remark} \label{BSD-rem}
There is an equality $L(A_f,s)=\prod_{\sigma:\Q_f\hookrightarrow\C}L(f^\sigma,s)$, where $L(A_f,s)$ is the $L$-function of $A_f$ and $f^\sigma$ is the conjugate of $f$ under a complex embedding $\sigma$ of $\Q_f$. It is conjectured that $r_\an(f^\sigma)=r_\an(f)$ for all such $\sigma$ (this equality is known to hold if $r_\an(f)\in\{0,1\}$, \emph{cf.} \cite[Ch. V, Corollary 1.3]{GZ}), and then the Birch--Swinnerton-Dyer conjecture leads to the expectation that the rank of $A_f(\Q)$ is divisible by $d_f$. In particular, $d_f$ is conjecturally the smallest non-zero rank that can be attained by $A_f(\Q)$.
\end{remark}

\subsubsection{Algebraic ranks and Shafarevich--Tate groups} \label{alg-sha-subsubsec}

As before, take an admissible weight $k\in U\cap\mathscr W_{N,\cla,4}$ with $k\equiv2\pmod{2(p-1)}$. As in \S \ref{sha-subsubsec}, given a number field $F$, denote by $\Sha_\fP(f_k^\flat/F)$ the $\fP$-primary Shafarevich--Tate group of $f_k^\flat$ over $F$, where $\fP$ is used as a shorthand for $\fP\cap\cO_{f_k^\flat}$ (\emph{cf.} Notation \ref{notation-rem}).

The next result asserts that the finiteness of the $p$-primary Shafarevich--Tate group of $f$ together with the minimal algebraic rank property from Assumption \ref{main-ass2} (\emph{cf.} Remark \ref{BSD-rem}) propagates to certain even weight forms in $\hf_U$.

\begin{theorem} \label{sha-thm}
$r_{\alg,\fP}(f_k^\flat/\Q)=1$ and $\#\Sha_\fP(f_k^\flat/\Q)<\infty$.
\end{theorem}

\begin{proof} Let $K$ be the imaginary quadratic field chosen in \S \ref{heegner-subsubsec}. Using Theorem \ref{heegner-nontrivial-prop} and \cite[Propositions 5.9]{Vigni}, one proceeds as in the proof of \cite[Theorem 5.26]{Vigni} to show that $r_{\alg,\fP}(f_k^\flat/\Q)=1$. Finally, $\Sha_\fP(f_k^\flat/K)$ is finite by \cite[Theorem 13.1]{Nek}, and then the finiteness of $\Sha_\fP(f_k^\flat/\Q)$ follows from \cite[Proposition 5.25]{Vigni}. \end{proof}

\begin{remark}
While \cite{Vigni} deals with the case of Hida (\emph{i.e.}, slope $0$) families, the results in \cite[\S 5.3]{Vigni}, like many others in \cite{Vigni}, are insensitive to the prime $p$ being ordinary or not (actually, they have nothing to do with families of modular forms; rather, they are statements on the arithmetic of a given newform).
\end{remark}

\begin{remark}
Strictly speaking, \cite[Proposition 5.25]{Vigni} concerns a specialization of a Hida family, but the arguments used in the proof of it apply \emph{verbatim} to any newform. 
\end{remark}

\begin{remark}
See \cite[Remark 5.23]{Vigni} for an alternative description of $\Sha_\fP(f_k^\flat/K)$ when $y_{k,K}$ is not torsion.
\end{remark}

The only condition on $p$ is that it belong to $\Xi_f$, whose complement in the set of prime numbers is finite, so Theorem \ref{sha-thm} implies the rank $1$ part of Theorem A.

\begin{remark}
The analogue for Hida families of Theorem \ref{sha-thm} is \cite[Theorem 5.26]{Vigni}.
\end{remark}

\begin{corollary} \label{sha-coro}
$r_{\alg,\fP}(f_k^\flat/\Q)=r_{\min}\bigl(\hf_U\bigr)$.
\end{corollary}

\begin{proof} Since $\varepsilon\bigl(\hf_U\bigr)=-1$, we have $r_{\min}\bigl(\hf_U\bigr)=1$, and the corollary follows at once from Theorem \ref{sha-thm}. \end{proof}

\section{Analytic ranks in Coleman families} \label{analytic-sec}

In this section, we prove results on the behaviour of analytic ranks of specializations in a Coleman family $\hf_U$ when $r_{\min}\bigl(\hf_U\bigr)\in\{0,1\}$. In particular, our results provide some evidence for Conjecture \ref{greenberg-conj} and, as a consequence, for the more general conjectures formulated by Greenberg in \cite{Greenberg-CRM}. As in Section \ref{sha-sec}, take a prime number $p\in\Xi_f$.

\subsection{The analytic rank one case} \label{analytic-rank-1-subsec}

Here we offer a result towards Conjecture \ref{greenberg-conj} when $r_{\min}\bigl(\hf_U\bigr)=1$. Notice that, as explained in \S \ref{root-numbers-subsubsec}, this equality is guaranteed by the fact that $r_\an(f)=1$ (\emph{cf}. Assumption \ref{main-ass}), which implies that $\varepsilon(f)=-1$.

\subsubsection{Results over $K$} \label{K-subsubsec}

Here we are interested in the analytic ranks of modular forms in our Coleman family $\hf_U$ after base change to the imaginary quadratic field $K$ that we chose in \S \ref{heegner-subsubsec}. As before, the Dirichlet character associated with $K$ will be denoted by $\chi_K$. 

Using arithmetic intersection theory \emph{\`a la} Gillet--Soul\'e (\cite{GS-1}), S.-W. Zhang defined in \cite{Zhang-heights} a pairing between certain CM cycles on Kuga--Sato varieties. Using this pairing, which we denote by ${\langle\cdot,\cdot\rangle}_{g,\GS}$, he proved a formula of Gross--Zagier type for the $L$-function of a higher weight modular form (\cite[Corollary 0.3.2]{Zhang-heights}; \emph{cf.} also \cite[Theorem 6.1]{Vigni}). In the present article, we apply Zhang's result to the specializations $f_k$ of $\hf_U$ (or, rather, to the newforms $f_k^\flat$).

Unfortunately, unlike what happens in weight $2$, the non-degeneracy of ${\langle\cdot,\cdot\rangle}_{g,\GS}$ for a $g$ of arbitrary (higher) weight is currently only conjectural; for example, it is a special case of \cite[Conjecture 2]{GS-2}. As a consequence, our results on analytic ranks will be conditional on the following

\begin{assumption} \label{main-ass3}
The pairing ${\langle\cdot,\cdot\rangle}_{f_k,\GS}$ is non-degenerate for each admissible $k\in U\cap\mathscr W_{N,\cla}^0$.
\end{assumption}

Note that a non-degeneracy condition of this sort is imposed also in \cite[Assumption 6.3]{Vigni} in the study of analytic ranks in Hida families. 

\begin{theorem} \label{quadratic-thm}
If $k\in U\cap\mathscr W_{N,\cla,4}$ is an admissible weight such that $k\equiv2\pmod{2(p-1)}$, then $r_\an(f_k^\flat/K)=1$.
\end{theorem}

\begin{proof} Let $k$ be as in the statement of the theorem. By \cite[Proposition 4.2]{Vigni}, we know that 
\begin{equation} \label{order-K-eq}
r_\an(f_k^\flat/K)\geq1.
\end{equation}   
By the same arguments as in the proof of \cite[Theorem 6.4]{Vigni}, if $r_\an(f_k^\flat/K)>1$, then \cite[Theorem 6.1]{Vigni} and Assumption \ref{main-ass3} force $y_{k,K}$ to be torsion, which contradicts Proposition \ref{heegner-nontrivial-prop}. Thus, it follows from \eqref{order-K-eq} that $r_\an(f_k^\flat/K)=1$. \end{proof}

\subsubsection{Proof of Theorem B, $r_\an=1$} \label{B-1-subsubsec}

Let $k\in U\cap\mathscr W_{N,\cla}^0$ be admissible. As in \S \ref{K-subsubsec}, let $K$ be the imaginary quadratic field, with associated Dirichlet character $\chi_K$, that was fixed in \S \ref{heegner-subsubsec}. The factorization $L(f_k^\flat/K,s)=L(f_k^\flat,s)\cdot L(f_k^\flat\otimes\chi_K,s)$ yields the equality 
\begin{equation} \label{analytic-ranks-eq}
r_\an(f^\flat_k/K)=r_\an(f^\flat_k)+r_\an(f^\flat_k\otimes\chi_K). 
\end{equation}
On the other hand, $r_\an(f^\flat_k/K)=1$ by Theorem \ref{quadratic-thm}, while $r_\an(f^\flat_k\otimes\chi_K)=0$ by condition (b) in \S \ref{heegner-subsubsec}. Therefore, it follows from \eqref{analytic-ranks-eq} that $r_\an(f^\flat_k)=1$, as was to be shown.\hfill\qedsymbol

\subsection{The analytic rank zero case} \label{analytic-0-subsec}

We prove a result in the direction of Conjecture \ref{greenberg-conj} when $r_{\min}\bigl(\hf_U\bigr)=0$ (\emph{cf.} Remark \ref{r-min-rem}). To achieve our goal, we will need to use an auxiliary imaginary quadratic twist of $\hf_U$ (\emph{cf.} \S \ref{B-0-subsubsec}): resorting to an argument of this kind is perhaps novel in this context. We remark that, as in the $r_{\min}\bigl(\hf_U\bigr)=1$ case (\emph{cf.} \S \ref{analytic-rank-1-subsec}), we impose a non-degeneracy condition on certain Gillet--Soul\'e height pairings; in Appendix \ref{appendix} we will sketch how one can use two-variable $p$-adic $L$-functions attached to $p$-adic Coleman families to prove a rank $0$ result in the vein of Theorem B without any extra assumption of the kind above (this is well known for Hida families: see, \emph{e.g.}, \cite[Theorem 7]{Howard-derivatives}).

\subsubsection{The newform $f$ of weight $2$ and choice of $p$} \label{newform-0-subsubsec}

As before, let $f\in S_2^\new(\Gamma_0(N))$ be a normalized newform of weight $2$ and square-free level $N$. Unlike what was done in \S \ref{analytic-assumption-subsubsec}, from here until the end of the paper we make the following

\begin{assumption} \label{main-0-ass}
$r_\an(f)=0$.
\end{assumption}

It follows that, in particular, $\varepsilon(f)=1$. By \cite[p. 543, Theorem, (i)]{BFH} (\emph{cf.} also \cite{MM-derivatives}), there is an imaginary quadratic field $K$, with associated Dirichlet character $\chi_K$, such that
\begin{itemize}
\item[(a)] the primes dividing $N$ split in $K$;
\item[(b)] $r_\an(f\otimes\chi_K)=1$.
\end{itemize}
Fix such a field $K$ and write $D_K$ for the discriminant of $K$. As in \S \ref{heegner-subsubsec},we consider the Heegner point $\alpha_K\in A_f(K)$. Combining Assumption \ref{main-0-ass} and condition (b), we deduce that $r_\an(f/K)=1$; by the Gross--Zagier formula, this is tantamount to $\alpha_K$ being non-torsion. Finally, we define a set $\Xi^0_f$ of prime numbers exactly as the set $\Xi_f$ from \S \ref{choice-subsubsec}.

Now put $M\defeq ND_K^2$ and $g\defeq f\otimes\chi_K\in S_2^\new(\Gamma_0(M))$, so that $r_\an(g)=1$ by condition (b). By \cite[p. 543, Theorem, (ii)]{BFH}, there is an imaginary quadratic field $K'$, with associated Dirichlet character $\chi_{K'}$, such that
\begin{itemize}
\item[(a')] the primes dividing $M$ split in $K'$;
\item[(b')] $r_\an(g\otimes\chi_{K'})=0$.
\end{itemize}
As before, consider the Heegner point $\alpha_{K'}\in A_g(K')$. Combining conditions (b) and (b'), we get $r_\an(g/K')=1$, hence $\alpha_{K'}$ is not torsion. On the other hand, $f$ has no complex multiplication, so $g$ has no complex multiplication as well. Using $g$ instead of $f$ and $K'$ in place of the field $K$ from \S \ref{heegner-subsubsec}, we define, as above, a set $\Xi_g$ of prime numbers by the same recipe as that for the set $\Xi_f$ introduced in \S \ref{choice-subsubsec}. Finally, let
\begin{equation} \label{omega-eq}
\Omega_f\defeq\Xi_f^0\cap\Xi_g; 
\end{equation}
then $\Omega_f$ excludes only finitely many primes.

Pick $p\in\Omega_f$ and, notation being as in \S \ref{stabilizations-subsubsec}, let $\hf_U=\sum_{n\geq1}a_n\bigl(\hf_U\bigr)q^n$ be a $p$-adic Coleman family (of tame level $N$) passing through $f^\sharp$. As usual, write $f_k$ for the specialization of $\hf_U$ at $k$.

\subsubsection{Proof of Theorem B, $r_\an=0$} \label{B-0-subsubsec}

Let us consider the quadratic twist
\[ \hf_U\otimes\chi_K\defeq\sum_{n\geq1}\chi_K(n)a_n(\f_U)q^n\in\Lambda_U[[q]], \]
which is a $p$-adic Coleman family over $U$ of tame level $M$ passing through $g$. Now let $k\in U\cap\mathscr W_{N,\cla}^0$ be admissible; there is an injection $\mathscr W_N\hookrightarrow\mathscr W_M$ (\emph{cf.} Remark \ref{weight-inclusion-rem}), so $k\in U\cap\mathscr W_{M,\cla}^0$. Clearly, the specialization of $\hf_U\otimes\chi_K$ at $k$ is $g_k\defeq f_k\otimes\chi_K$. Moreover, a direct computation shows that $(f_k^\flat\otimes\chi_K)^\sharp=g_k$, hence $g_k^\flat=f_k^\flat\otimes\chi_K$.

In analogy with what we did in \S \ref{K-subsubsec} in analytic rank $1$, we make the following

\begin{assumption} \label{main-ass4}
\begin{enumerate}
\item The pairing ${\langle\cdot,\cdot\rangle}_{f_k,\GS}$ is non-degenerate for each admissible $k\in U\cap\mathscr W_{N,\cla}^0$.
\item The pairing ${\langle\cdot,\cdot\rangle}_{g_k,\GS}$ is non-degenerate for each admissible $k\in U\cap\mathscr W_{M,\cla}^0$.
\end{enumerate}
\end{assumption}

Combining part (1) of Assumption \ref{main-ass4} with the analogue of Theorem \ref{heegner-nontrivial-prop} for our current choice of $\hf_U$, we deduce that $r_\an(f^\flat_k/K)=1$. On the other hand, Theorem \ref{heegner-nontrivial-prop} holds for the Coleman family $\hf_U\otimes\chi_K$ and its specialization $g_k$ as well. Therefore, since $r_\an(g)=1$, the $r_\an=1$ part of Theorem B (\emph{cf.} \S \ref{B-1-subsubsec}) applied to $\hf_U\otimes\chi_K$, which we can invoke thanks to part (2) of Assumption \ref{main-ass4}, ensures that $r_\an(g^\flat_k)=1$. Since $g^\flat_k=f^\flat_k\otimes\chi_K$, the equality
\[ r_\an(f^\flat_k/K)=r_\an(f^\flat_k)+r_\an(g^\flat_k) \]
gives $r_\an(f^\flat_k)=0$, which is equivalent, by \cite[Lemma 2.9]{Vigni}, to $r_\an(f_k)=0$.\hfill\qedsymbol

\begin{remark} \label{r-min-rem}
Under our assumptions on analytic ranks, congruence \eqref{analytic-congruence-eq} gives $\varepsilon(f_k^\flat)=1$. It follows from \eqref{epsilon-omega-eq} that $\varepsilon\bigl(\hf_U\bigr)=\omega_N\bigl(\hf_U\bigr)=1$, and then $r_{\min}\bigl(\hf_U\bigr)=0$.
\end{remark}

\begin{remark}
In Appendix \ref{appendix} we explain how two-variable $p$-adic $L$-functions attached to $\hf_U$ can be used to prove a refined version of the rank $0$ part of Theorem B for which we can drop the congruence condition on the weights and the non-degeneracy assumption on height pairings is not necessary. From our perspective, our strategy to prove this analytic rank $0$ result remains interesting because it bypasses any consideration whatsoever involving $p$-adic $L$-functions.
\end{remark}

\section{Shafarevich--Tate groups in Coleman families: the rank zero case} \label{rank-zero-sec}

In this section, the weight $2$ newform $f$ and the imaginary quadratic field $K$ are chosen as in \S \ref{newform-0-subsubsec}; in particular, we work under Assumption \ref{main-0-ass}. 

\subsection{On Shafarevich--Tate groups in families} \label{sha-subsec2}

We want to prove the counterpart in the rank $0$ case of Theorem \ref{sha-thm}.

\subsubsection{Wan's converse theorem}

The result below slightly generalizes a converse theorem for elliptic curves due to Skinner--Urban (\cite[Theorem 2, (b)]{SU}). Although we were not able to find a reference where its proof is spelled out in detail, this statement is certainly well known to the experts.

\begin{theorem}[Castella--\c{C}iperiani--Skinner--Sprung] \label{ccss-thm}
If $\mathrm{rank}_\Z\, A_f(\Q)=0$ and $\Sha_{\p}(A_f/\Q)$ is finite, then $r_\an(f)=0$. Conversely, if $r_\an(f)=0$, then $\mathrm{rank}_\Z\, A_f(\Q)=0$ and $\Sha_{p^\infty}(A_f/\Q)$ is finite.
\end{theorem}

\begin{proof}[Sketch of proof.] Using \cite[Theorem A]{CCSS}, one obtains an extension of \cite[Theorem 2, (b)]{SU} to our context, from which the theorem follows. \end{proof}

\subsubsection{Algebraic ranks and Shafarevich--Tate groups}

We turn to our main algebraic result in rank $0$ (\emph{cf.} Theorem \ref{sha-thm} for the rank $1$ case). Thanks to Theorem \ref{ccss-thm}, we can reformulate Assumption \ref{main-0-ass} as 

\begin{assumption} \label{main-0-ass2}
$\rank_\Z A_f(\Q)=0$ and $\#\Sha_\p(A_f/\Q)<\infty$.
\end{assumption}

Now let $p\in\Omega_f$, where $\Omega_f$ is the set of primes defined in \eqref{omega-eq}, and take an admissible $k\in U\cap\mathscr W_{N,\cla,4}$ with $k\equiv2\pmod{2(p-1)}$. 

\begin{theorem} \label{sha-thm2}
$r_{\alg,\fP}(f_k^\flat/\Q)=0$ and $\#\Sha_\fP(f_k^\flat/\Q)<\infty$.
\end{theorem}

\begin{proof} Using an analogue of Theorem \ref{heegner-nontrivial-prop} in our current setting, one can argue exactly as in the proof of \cite[Theorem 7.4]{Vigni}. \end{proof}

The set $\Omega_f$ excludes only finitely many prime numbers, so Theorem \ref{sha-thm2} implies the rank $0$ part of Theorem A.

\begin{remark}
The counterpart for Hida families of Theorem \ref{sha-thm2} is \cite[Theorem 7.4]{Vigni}.
\end{remark}

\begin{corollary} \label{sha-coro2}
$r_{\alg,\fP}(f_k^\flat/\Q)=r_{\min}\bigl(\hf_U\bigr)$.
\end{corollary}

\begin{proof} Since $\varepsilon\bigl(\hf_U\bigr)=1$, we have $r_{\min}\bigl(\hf_U\bigr)=0$, and then the corollary is an immediate consequence of Theorem \ref{sha-thm2}. \end{proof}

\appendix

\section{Two-variable $p$-adic $L$-functions and the rank zero case} \label{appendix}

We briefly explain how the two-variable $p$-adic $L$-functions attached to $p$-adic Coleman families by Bella\"iche--Pollack--Stevens (see, \emph{e.g.}, \cite{bellaiche}) and Panchishkin (\cite{Panchishkin-Inventiones}) can be used to obtain a refinement of the rank $0$ part of Theorem B for which we can drop the congruence condition on the weights and the non-degeneracy of height pairings is not needed. While the corresponding statement for Hida families is well known (see, \emph{e.g.}, \cite[Theorem 7]{Howard-derivatives}), no result of this sort for Coleman families seems to be available, so Theorem \ref{main-appendix-thm} fills a gap in the literature. Throughout this appendix, we freely employ terminology and notation from the main body of the paper.

\subsection{Two-variable $p$-adic $L$-functions}

We work in somewhat greater generality than in the rest of the article.

\subsubsection{The setting}

Let $p$ be a prime number, let $N\geq1$ be an integer such that $p\nmid N$ and take a newform $f\in S^\new_{k_0}\bigl(\Gamma_1(N),\omega^i\bigr)$, where $k_0\geq2$ is an integer, $\omega:\F_p^\times\rightarrow\Bmu_{p-1}$ is the Teichm\"uller character and $i\in\Z$. The two roots of the Hecke polynomial of $f$ at $p$ (in the sense of part (1) of Definition \ref{noble-def}; \emph{cf.} also \S \ref{hecke-polynomial-subsubsec} for the $k_0=2$ case) will be denoted by $\alpha$ and $\beta$. We assume that the $p$-stabilization $f_\beta$ of $f$ is noble and has non-critical slope (\emph{i.e.}, $v_p(\beta)<k_0-1$). As before, we consider a $p$-adic Coleman family $\f_U$ passing through $f$ (or, rather, $f_\beta$) defined over some wide open disc $U$ of the weight space $\mathscr W_N$.

\subsubsection{Two-variable $p$-adic $L$-functions and interpolation}

By \cite[Theorem 3]{bellaiche}, for a choice of sign $\pm$ there is a two-variable analytic $p$-adic $L$-function $L_p^\pm(x,\sigma)$, where $x\in U$ and $\sigma:\Z_p^\times\rightarrow\C_p^\times$ is a continuous character ($\C_p$ is, as customary, the completion of $\bar\Q_p$). Let $k\in U\cap\mathscr W_{N,\cla}$ be admissible, let $\phi$ be a character of finite order of $\Z_p^\times$ with $\phi(-1)=\pm1$ and let $j$ be an integer such that $0\leq j\leq k-2$; denote by $\phi t^j$ the character $t \mapsto \phi(t)t^j$. Finally, write $\bar\phi$ for the inverse (\emph{i.e.}, conjugate) character to $\phi$ (see, \emph{e.g.}, \cite[Ch. I, \S 13]{MTT} for further details and standard conventions on $p$-adic characters, which we shall tacitly use in the lines below). The $p$-adic $L$-function $L_p^\pm$ satisfies (\cite[eq. (4) and Theorem 3]{bellaiche}) the interpolation formula
\begin{equation} \label{eq:interpolationL}
L_p^\pm(k,\phi t^j) = c \cdot e_p(\beta_k,\phi t^j) \cdot \frac{m^{j+1}}{(-2\pi i )^j} \cdot \frac{j!}{\tau(\bar\phi)} \cdot \frac{L(f_k^\flat,\bar\phi,j+1)}{\Omega_{f^\flat_k}^\pm},
\end{equation}
where
\begin{itemize}
  \item $c$ is a non-zero scalar depending only on $k$,
  \item $m=p^\nu M$, with $p\nmid M$, is the conductor of $\phi$,
  \item $L(f^\flat_k,\bar\phi,s)$ is the complex $L$-function of $f_k^\flat$ twisted by $\bar\phi$,
  \item $\Omega_{f^\flat_k}^\pm$ are the two archimedean periods of $f^\flat_k$,
  \item $\tau(\bar\phi)$ is the Gauss sum of $\bar\phi$,
  \item $\displaystyle{e_p(\beta_k,\phi t^j)=\frac{1}{(\beta_k)^\nu}\cdot\biggl(1-\frac{\bar\phi(p)\chi_{f^\flat_k}(p)p^{k-j}}{\alpha_k}\biggr)\cdot\biggl(1-\frac{\phi(p)p^j}{\alpha_k}\biggr)}$; here $\alpha_k,\beta_k$ are the roots of the Hecke polynomial of $f_k^\flat$ at $p$.
\end{itemize}
In light of the interpolation formula recalled in \eqref{eq:interpolationL}, we say that $L_p^\pm$ is the two-variable $p$-adic $L$-function attached to $\f_U$. 

\begin{remark}
The reader is invited to compare \eqref{eq:interpolationL} with the formula in \cite[Ch. I, \S 14, Proposition]{MTT}.
\end{remark}

\begin{remark}
For an alternative approach to the construction of the two-variable $p$-adic $L$-function, see \cite[Theorem 0.5]{Panchishkin-Inventiones}.
\end{remark}

\subsubsection{Arithmetic primes and critical character}

Let us consider the subgroup $\Gamma\defeq1+p\Z_p$ of principal units of $\Z_p^\times$, set $\Lambda\defeq\cO_L[[\Gamma]]$ and fix isomorphisms
\begin{equation} \label{Lambda-isom-eq}
\Lambda\overset\simeq\longrightarrow\cO_L[[T]]\overset\simeq\longrightarrow\Lambda_U
\end{equation}
of $\cO_L$-algebras. A continuous homomorphism $\eta:\Lambda_U\rightarrow\bar\Q_p$ is \emph{arithmetic} if the composition $\Gamma\hookrightarrow\Lambda_U^\times\xrightarrow\eta\bar\Q_p^\times$ with the injection $\Gamma\hookrightarrow\Lambda_U^\times$ induced by \eqref{Lambda-isom-eq} has the form $\gamma\mapsto\psi(\gamma)\gamma^{k-2}$ for an integer $k\geq2$ (the \emph{weight} of $\eta$) and a finite order character $\psi$ of $\Gamma$ (the \emph{wild character} of $\eta$). An \emph{arithmetic prime} of $\Lambda_U$ is the kernel of an arithmetic morphism, and the weight $k_\wp$ and the wild character $\psi_\wp$ of an arithmetic prime are defined to be the weight and the wild character, respectively, of the corresponding morphism. The Coleman family $\f_U$ admits specializations $f_\wp$ at the arithmetic primes $\wp$ of $\Lambda_U$ that are defined \emph{mutatis mutandis} as in the case of Hida families (see, \emph{e.g.}, \cite[\S 2.5]{Vigni}). It turns out that if $\wp$ is an arithmetic prime and $\Lambda_{U,\wp}$ is the localization of $\Lambda_U$ at $\wp$, then the field $\mathcal F_\wp\defeq\Lambda_{U,\wp}/\wp\Lambda_{U,\wp}$ is a finite extension of $\Q_p$ containing the Fourier coefficients of $f_\wp$.

Similarly to what is done in \cite{Howard-derivatives} and \cite{Howard-Inv}, we consider a \emph{critical character}
\[ \theta:\Z_p^\times\longrightarrow\Lambda_U^\times \]
such that $\theta^2(t)=[t]$ for $t\in\Z_p^\times$, where the map $[\cdot]:\Z_p^\times\rightarrow\Lambda_U^\times$ has the property that if $\wp$ is an arithmetic prime of $\Lambda_U$, then the composition $[\cdot]_\wp$ of $[\cdot]$ with the canonical map $\Lambda_U^\times\rightarrow\mathcal F_\wp^\times$ is given by $[t]_\wp=\omega(t)^{k_0+i-k_\wp}\psi_\wp(t)t^{k_\wp-2}$. Analogously, denote by $\theta_\wp$ the composition of $\theta$ with $\Lambda_U^\times\rightarrow\mathcal F_\wp^\times$. Now observe that if $k\in U\cap\mathscr W_{N,\cla}^0$ is admissible, then the character of $f_k$ is $\omega^{k_0+i-k}$; as a consequence, $\theta_k(t)=\omega(t)^{(k_0+i-k)/2} t^{(k-2)/2}$ and the modular form
\[ g_k\defeq f_k\otimes\omega^{-(k_0+i-k)/2}\in S_k(\Gamma_0(Np^2)) \]
has trivial character (see, \emph{e.g.}, \cite[Ch. III, Proposition 17, (b)]{Koblitz}).

%\(\Lp^\pm(k,\theta_k)\)
Finally, following Mazur--Tate--Teitelbaum (\emph{cf.} \cite[Ch. I, \S 15]{MTT}), we say that an admissible $k\in U\cap\mathscr W^0_{N,\cla}$ is a \emph{trivial zero} of $L_p^\pm$ if $e_p(\beta_k,\theta_k)=0$.

\begin{proposition} \label{coro:ranL}
If an admissible $k\in U\cap\mathscr W_{N,\cla}^0$ is not a trivial zero, then $r_\an(g_k)=0$ if and only if $L_p^\pm(k,\theta_k)\neq0$. 
\end{proposition}

\begin{proof} Immediate from formula \eqref{eq:interpolationL}. \end{proof}

\subsection{The analytic rank zero case} 

Here we prove the refinement of the rank $0$ part of Theorem B that we are interested in.

\subsubsection{A lemma on analytic functions on $U$}

Before turning to the main theorem of the appendix, we need an auxiliary result. Denote by $\mathscr R_U$ the ring of analytic functions on $U$.

\begin{lemma} \label{lemm:zero_affinoid}
Let $h\in\mathscr R_U$ be non-torsion. Then $h(x)\neq0$ for all but finitely many $x\in U$.
\end{lemma}

\begin{proof} Since $\mathscr R_U$ is noetherian, this is \cite[Lemma 2.1.7]{Howard-Inv}. \end{proof}

\subsubsection{The analytic rank zero case}

To begin with, let us set
\[ \mathscr U\defeq\bigl\{k\in U\cap\mathscr W_{N,\cla}^0\mid\text{$k$ is admissible and not a trivial zero of $L_p^\pm$}\bigr\}. \]
The result we are about to prove is an analogue in the finite slope setting of the analytic part of \cite[Theorem 7]{Howard-derivatives}.

\begin{theorem} \label{main-appendix-thm}
The following statements are equivalent:
\begin{enumerate}
\item there is $k\in\mathscr U$ such that $r_\an(g_k)=0$;
\item $r_\an(g_k)=0$ for all but finitely many $k\in\mathscr U$.
\end{enumerate} 
\end{theorem}

Note that the function
\[ L_p^\pm(\cdot,\theta):U\longrightarrow\C_p,\quad x\longmapsto L_p^\pm(x,\theta_x) \]
is analytic over $U$.

\begin{proof}[Proof of Theorem \ref{main-appendix-thm}] Assume there is $k\in\mathscr U$ such that $r_\an(g_k)=0$. Then, by Proposition \ref{coro:ranL}, we have $L_p^\pm(k,\theta_k)\neq0$, hence $L_p^\pm(\cdot,\theta)\neq0$. Since $\mathscr R_U$ is an integral domain, we deduce that $L_p^\pm(\cdot,\theta)$ is not torsion, and the theorem follows from Lemma \ref{lemm:zero_affinoid}. \end{proof}

\begin{remark}
Under the assumptions of Theorem B (\emph{i.e.}, $k_0=2$, $i=0$), if we add the condition $k\equiv2\pmod{2(p-1)}$, then Theorem \ref{main-appendix-thm} provides indeed a refinement of the rank $0$ case of Theorem B.
\end{remark}

\bibliographystyle{amsplain}
\bibliography{Coleman}

\providecommand{\bysame}{\leavevmode\hbox to3em{\hrulefill}\thinspace}
\providecommand{\MR}{\relax\ifhmode\unskip\space\fi MR }
\providecommand{\MRhref}[2]{%
  \href{http://www.ams.org/mathscinet-getitem?mr=#1}{#2}
}
\providecommand{\href}[2]{#2}
\begin{thebibliography}{10}

\bibitem{Andre}
Y.~Andr\'{e}, \emph{Une introduction aux motifs (motifs purs, motifs mixtes,
  p\'{e}riodes)}, Panoramas et Synth\`eses, vol.~17, Soci\'{e}t\'{e}
  Math\'{e}matique de France, Paris, 2004.

\bibitem{AIS}
F.~Andreatta, A.~Iovita, and G.~Stevens, \emph{Overconvergent
  {E}ichler--{S}himura isomorphisms}, J. Inst. Math. Jussieu \textbf{14}
  (2015), no.~2, 221--274.

\bibitem{AL}
A.~O.~L. Atkin and W.~C.~W. Li, \emph{Twists of newforms and pseudo-eigenvalues
  of {$W$}-operators}, Invent. Math. \textbf{48} (1978), no.~3, 221--243.

\bibitem{bellaiche}
J.~Bella\"{\i}che, \emph{Critical {$p$}-adic {$L$}-functions}, Invent. Math.
  \textbf{189} (2012), no.~1, 1--60.

\bibitem{BD-rational}
M.~Bertolini and H.~Darmon, \emph{Hida families and rational points on elliptic
  curves}, Invent. Math. \textbf{168} (2007), no.~2, 371--431.

\bibitem{BDP}
M.~Bertolini, H.~Darmon, and K.~Prasanna, \emph{Generalized {H}eegner cycles
  and {$p$}-adic {R}ankin {$L$}-series}, Duke Math. J. \textbf{162} (2013),
  no.~6, 1033--1148.

\bibitem{BK}
S.~Bloch and K.~Kato, \emph{{$L$}-functions and {T}amagawa numbers of motives},
  The {G}rothendieck {F}estschrift, {V}ol.\ {I}, Progr. Math., vol.~86,
  Birkh\"auser Boston, Boston, MA, 1990, pp.~333--400.

\bibitem{BFH}
D.~Bump, S.~Friedberg, and J.~Hoffstein, \emph{Nonvanishing theorems for
  {$L$}-functions of modular forms and their derivatives}, Invent. Math.
  \textbf{102} (1990), no.~3, 543--618.

\bibitem{BL-bis}
K.~B\"uy\"ukboduk and A.~Lei, \emph{Interpolation of generalized {H}eegner
  cycles in {C}oleman families}, arXiv:1907.04086v2.

\bibitem{BL}
\bysame, \emph{Interpolation of generalized {H}eegner cycles in {C}oleman
  families}, J. Lond. Math. Soc. (2) \textbf{104} (2021), no.~4, 1682--1716.

\bibitem{buzzard-slopes}
K.~Buzzard, \emph{Questions about slopes of modular forms}, Ast\'{e}risque
  (2005), no.~298, 1--15.

\bibitem{Buzzard-eigenvarieties}
\bysame, \emph{Eigenvarieties}, {$L$}-functions and {G}alois representations,
  London Math. Soc. Lecture Note Ser., vol. 320, Cambridge Univ. Press,
  Cambridge, 2007, pp.~59--120.

\bibitem{BG-slopes}
K.~Buzzard and T.~Gee, \emph{Slopes of modular forms}, Families of automorphic
  forms and the trace formula, Simons Symp., Springer, Cham, 2016, pp.~93--109.

\bibitem{CS}
F.~Calegari and N.~T. Sardari, \emph{Vanishing {F}ourier coefficients of
  {H}ecke eigenforms}, Math. Ann. \textbf{381} (2021), no.~3-4, 1197--1215.

\bibitem{CasHeeg}
F.~Castella, \emph{On the {$p$}-adic variation of {H}eegner points}, J. Inst.
  Math. Jussieu \textbf{19} (2020), no.~6, 2127--2164.

\bibitem{CCSS}
F.~Castella, M.~\c{C}iperiani, C.~Skinner, and F.~Sprung, \emph{On the
  {I}wasawa main conjectures for modular forms at non-ordinary primes},
  arXiv:1804.10993v2, submitted.

\bibitem{Coleman}
R.~F. Coleman, \emph{{$p$}-adic {B}anach spaces and families of modular forms},
  Invent. Math. \textbf{127} (1997), no.~3, 417--479.

\bibitem{CE}
R.~F. Coleman and B.~Edixhoven, \emph{On the semi-simplicity of the
  {$U_p$}-operator on modular forms}, Math. Ann. \textbf{310} (1998), no.~1,
  119--127.

\bibitem{ColMaz}
R.~F. Coleman and B.~Mazur, \emph{The eigencurve}, Galois representations in
  arithmetic algebraic geometry ({D}urham, 1996), London Math. Soc. Lecture
  Note Ser., vol. 254, Cambridge Univ. Press, Cambridge, 1998, pp.~1--113.

\bibitem{CLM}
A.~Conti, J.~Lang, and A.~Medvedovsky, \emph{Big images of two-dimensional
  pseudorepresentations}, Math. Ann. \textbf{385} (2023), no.~3-4, 1085--1179.

\bibitem{Del-Bourbaki}
P.~Deligne, \emph{Formes modulaires et repr\'esentations {$\ell$}-adiques},
  S\'eminaire {B}ourbaki. {V}ol. 1968/69: {E}xpos\'es 347--363, Lecture Notes
  in Math., vol. 175, Springer, Berlin, 1971, pp.~139--172.

\bibitem{disegni}
D.~Disegni, \emph{The universal {$p$}-adic {G}ross--{Z}agier formula}, Invent.
  Math. \textbf{230} (2022), no.~2, 509--649.

\bibitem{fischman}
A.~Fischman, \emph{On the image of {$\Lambda$}-adic {G}alois representations},
  Ann. Inst. Fourier (Grenoble) \textbf{52} (2002), no.~2, 351--378.

\bibitem{Flach}
M.~Flach, \emph{A generalisation of the {C}assels--{T}ate pairing}, J. Reine
  Angew. Math. \textbf{412} (1990), 113--127.

\bibitem{GS-1}
H.~Gillet and C.~Soul\'{e}, \emph{Arithmetic intersection theory}, Inst. Hautes
  \'{E}tudes Sci. Publ. Math. (1990), no.~72, 93--174 (1991).

\bibitem{GS-2}
\bysame, \emph{Arithmetic analogs of the standard conjectures}, Motives
  ({S}eattle, {WA}, 1991), Proc. Sympos. Pure Math., vol.~55, Amer. Math. Soc.,
  Providence, RI, 1994, pp.~129--140.

\bibitem{gouvea-slopes}
F.~Q. Gouv\^{e}a, \emph{Where the slopes are}, J. Ramanujan Math. Soc.
  \textbf{16} (2001), no.~1, 75--99.

\bibitem{GSS}
M.~Greenberg, M.~A. Seveso, and S.~Shahabi, \emph{Modular {$p$}-adic
  {$L$}-functions attached to real quadratic fields and arithmetic
  applications}, J. Reine Angew. Math. \textbf{721} (2016), 167--231.

\bibitem{Greenberg-CRM}
R.~Greenberg, \emph{Elliptic curves and {$p$}-adic deformations}, Elliptic
  curves and related topics, CRM Proc. Lecture Notes, vol.~4, Amer. Math. Soc.,
  Providence, RI, 1994, pp.~101--110.

\bibitem{GP}
B.~H. Gross and J.~A. Parson, \emph{On the local divisibility of {H}eegner
  points}, Number theory, analysis and geometry, Springer, New York, 2012,
  pp.~215--241.

\bibitem{GZ}
B.~H. Gross and D.~B. Zagier, \emph{Heegner points and derivatives of
  {$L$}-series}, Invent. Math. \textbf{84} (1986), no.~2, 225--320.

\bibitem{hida-modular}
H.~Hida, \emph{Modular forms and {G}alois cohomology}, Cambridge Studies in
  Advanced Mathematics, vol.~69, Cambridge University Press, Cambridge, 2000.

\bibitem{Howard-derivatives}
B.~Howard, \emph{Central derivatives of {$L$}-functions in {H}ida families},
  Math. Ann. \textbf{339} (2007), no.~4, 803--818.

\bibitem{Howard-Inv}
\bysame, \emph{Variation of {H}eegner points in {H}ida families}, Invent. Math.
  \textbf{167} (2007), no.~1, 91--128.

\bibitem{JLZ}
D.~Jetchev, D.~Loeffler, and S.~L. Zerbes, \emph{Heegner points in {C}oleman
  families}, Proc. Lond. Math. Soc. (3) \textbf{122} (2021), no.~1, 124--152.

\bibitem{jochnowitz-TAMS}
N.~Jochnowitz, \emph{Congruences between systems of eigenvalues of modular
  forms}, Trans. Amer. Math. Soc. \textbf{270} (1982), no.~1, 269--285.

\bibitem{Kato}
K.~Kato, \emph{{$p$}-adic {H}odge theory and values of zeta functions of
  modular forms}, Ast\'{e}risque (2004), no.~295, ix, 117--290.

\bibitem{KLZ}
G.~Kings, D.~Loeffler, and S.~L. Zerbes, \emph{Rankin--{E}isenstein classes and
  explicit reciprocity laws}, Camb. J. Math. \textbf{5} (2017), no.~1, 1--122.

\bibitem{Knapp}
A.~W. Knapp, \emph{Elliptic curves}, Mathematical Notes, vol.~40, Princeton
  University Press, Princeton, NJ, 1992.

\bibitem{Koblitz}
N.~Koblitz, \emph{Introduction to elliptic curves and modular forms}, second
  ed., Graduate Texts in Mathematics, vol.~97, Springer-Verlag, New York, 1993.

\bibitem{Kol-Log}
V.~A. Kolyvagin and D.~Y. Logach\"{e}v, \emph{Finiteness of the
  {S}hafarevich--{T}ate group and the group of rational points for some modular
  abelian varieties}, Leningrad Math. J. \textbf{1} (1990), no.~5, 1229--1253.

\bibitem{kowalski}
E.~Kowalski, \emph{An introduction to the representation theory of groups},
  Graduate Studies in Mathematics, vol. 155, American Mathematical Society,
  Providence, RI, 2014.

\bibitem{LZ}
D.~Loeffler and S.~L. Zerbes, \emph{Rankin--{E}isenstein classes in {C}oleman
  families}, Res. Math. Sci. \textbf{3} (2016), Paper No. 29, 53.

\bibitem{LV}
M.~Longo and S.~Vigni, \emph{A refined {B}eilinson--{B}loch conjecture for
  motives of modular forms}, Trans. Amer. Math. Soc. \textbf{369} (2017),
  no.~10, 7301--7342.

\bibitem{MTT}
B.~Mazur, J.~Tate, and J.~Teitelbaum, \emph{On {$p$}-adic analogues of the
  conjectures of {B}irch and {S}winnerton-{D}yer}, Invent. Math. \textbf{84}
  (1986), no.~1, 1--48.

\bibitem{Milne}
J.~S. Milne, \emph{\'{E}tale cohomology}, Princeton Mathematical Series,
  vol.~33, Princeton University Press, Princeton, N.J., 1980.

\bibitem{Milne-AV}
\bysame, \emph{Abelian varieties}, Arithmetic geometry ({S}torrs, {C}onn.,
  1984), Springer, New York, 1986, pp.~103--150.

\bibitem{Miyake}
T.~Miyake, \emph{Modular forms}, Springer-Verlag, Berlin, 1989.

\bibitem{MM-derivatives}
M.~R. Murty and V.~K. Murty, \emph{Mean values of derivatives of modular
  {$L$}-series}, Ann. of Math. (2) \textbf{133} (1991), no.~3, 447--475.

\bibitem{Nek}
J.~Nekov{\'a}{\v{r}}, \emph{Kolyvagin's method for {C}how groups of
  {K}uga--{S}ato varieties}, Invent. Math. \textbf{107} (1992), no.~1, 99--125.

\bibitem{Nek3}
\bysame, \emph{{$p$}-adic {A}bel--{J}acobi maps and {$p$}-adic heights}, The
  arithmetic and geometry of algebraic cycles ({B}anff, {AB}, 1998), CRM Proc.
  Lecture Notes, vol.~24, Amer. Math. Soc., Providence, RI, 2000, pp.~367--379.

\bibitem{NP}
J.~Nekov{\'a}{\v{r}} and A.~Plater, \emph{On the parity of ranks of {S}elmer
  groups}, Asian J. Math. \textbf{4} (2000), no.~2, 437--497.

\bibitem{NSW}
J.~Neukirch, A.~Schmidt, and K.~Wingberg, \emph{Cohomology of number fields},
  second ed., Grundlehren der mathematischen Wissenschaften, vol. 323,
  Springer-Verlag, Berlin, 2008.

\bibitem{niziol}
W.~Nizio\l, \emph{On the image of {$p$}-adic regulators}, Invent. Math.
  \textbf{127} (1997), no.~2, 375--400.

\bibitem{ota-JNT}
K.~Ota, \emph{Big {H}eegner points and generalized {H}eegner cycles}, J. Number
  Theory \textbf{208} (2020), 305--334.

\bibitem{Panchishkin-Inventiones}
A.~A. Panchishkin, \emph{Two variable {$p$}-adic {$L$} functions attached to
  eigenfamilies of positive slope}, Invent. Math. \textbf{154} (2003), no.~3,
  551--615.

\bibitem{ribet-annals}
K.~A. Ribet, \emph{Endomorphisms of semi-stable abelian varieties over number
  fields}, Ann. of Math. (2) \textbf{101} (1975), no.~3, 555--562.

\bibitem{ribet}
\bysame, \emph{Galois representations attached to eigenforms with
  {N}ebentypus}, Modular functions of one variable, {V} ({P}roc. {S}econd
  {I}nternat. {C}onf., {U}niv. {B}onn, {B}onn, 1976), Springer, Berlin, 1977,
  pp.~17--51. Lecture Notes in Math., Vol. 601.

\bibitem{ribet-twists}
\bysame, \emph{Twists of modular forms and endomorphisms of abelian varieties},
  Math. Ann. \textbf{253} (1980), no.~1, 43--62.

\bibitem{ribet2}
\bysame, \emph{On {$l$}-adic representations attached to modular forms. {II}},
  Glasgow Math. J. \textbf{27} (1985), 185--194.

\bibitem{saito}
T.~Saito, \emph{Modular forms and {$p$}-adic {H}odge theory}, Invent. Math.
  \textbf{129} (1997), no.~3, 607--620.

\bibitem{saito2}
\bysame, \emph{Weight-monodromy conjecture for {$\ell$}-adic representations
  associated to modular forms}, The arithmetic and geometry of algebraic cycles
  ({B}anff, {AB}, 1998), NATO Sci. Ser. C Math. Phys. Sci., vol. 548, Kluwer
  Acad. Publ., Dordrecht, 2000, pp.~427--431. \MR{1744955}

\bibitem{Scholl}
A.~J. Scholl, \emph{Motives for modular forms}, Invent. Math. \textbf{100}
  (1990), no.~2, 419--430.

\bibitem{shimura}
G.~Shimura, \emph{Introduction to the arithmetic theory of automorphic
  functions}, Publications of the Mathematical Society of Japan, No. 11.
  Iwanami Shoten, Publishers, Tokyo; Princeton University Press, Princeton,
  N.J., 1971, Kan\^{o} Memorial Lectures, No. 1.

\bibitem{Skinner}
C.~Skinner, \emph{A converse to a theorem of {G}ross, {Z}agier, and
  {K}olyvagin}, Ann. of Math. (2) \textbf{191} (2020), no.~2, 329--354.

\bibitem{SU}
C.~Skinner and E.~Urban, \emph{The {I}wasawa {M}ain {C}onjectures for {$\rm
  GL_2$}}, Invent. Math. \textbf{195} (2014), no.~1, 1--277.

\bibitem{sweeting}
N.~Sweeting, \emph{Kolyvagin's conjecture and patched {E}uler systems in
  anticyclotomic {I}wasawa theory}, arXiv:2012.11771v2, submitted.

\bibitem{Tate-BSD}
J.~Tate, \emph{On the conjectures of {B}irch and {S}winnerton-{D}yer and a
  geometric analog}, Dix expos\'{e}s sur la cohomologie des sch\'{e}mas, Adv.
  Stud. Pure Math., vol.~3, North-Holland, Amsterdam, 1968, pp.~189--214.

\bibitem{Tate}
\bysame, \emph{Relations between {$K_{2}$} and {G}alois cohomology}, Invent.
  Math. \textbf{36} (1976), no.~1, 257--274.

\bibitem{vatsal-quebec}
V.~Vatsal, \emph{Integral periods for modular forms}, Ann. Math. Qu\'{e}.
  \textbf{37} (2013), no.~1, 109--128.

\bibitem{Vigni}
S.~Vigni, \emph{On {S}hafarevich--{T}ate groups and analytic ranks in families
  of modular forms, {I}. {H}ida families}, Ann. Sc. Norm. Super. Pisa Cl. Sci.
  (5), to appear, arXiv:2001.04310v4.

\bibitem{Waldspurger}
J.-L. Waldspurger, \emph{Sur les valeurs de certaines fonctions {$L$}
  automorphes en leur centre de sym\'{e}trie}, Compos. Math. \textbf{54}
  (1985), no.~2, 173--242.

\bibitem{Zhang-heights}
S.-W. Zhang, \emph{Heights of {H}eegner cycles and derivatives of
  {$L$}-series}, Invent. Math. \textbf{130} (1997), no.~1, 99--152.

\bibitem{zhang-selmer}
W.~Zhang, \emph{Selmer groups and the indivisibility of {H}eegner points},
  Camb. J. Math. \textbf{2} (2014), no.~2, 191--253.

\end{thebibliography}

\end{document}